\documentclass{amsart}

\usepackage{amssymb}
\usepackage{enumitem}

\usepackage{graphicx}

\newtheorem{theorem}{Theorem}[section]
\newtheorem{lemma}[theorem]{Lemma}
\newtheorem{proposition}[theorem]{Proposition}
\newtheorem{corollary}[theorem]{Corollary}

\theoremstyle{definition}
\newtheorem{definition}[theorem]{Definition}

\theoremstyle{remark}
\newtheorem{remark}[theorem]{Remark}

\numberwithin{equation}{section}

\newcommand{\N}{\mathbb{N}}
\newcommand{\R}{\mathbb{R}}
\def\phi{{\varphi}}
\def\epsilon{{\varepsilon}}

\DeclareMathOperator{\graph}{graph}

\def\Hinrichtung{``\ensuremath{\Rightarrow}'':\text{ }}
\def\Rueckrichtung{``\ensuremath{\Leftarrow}'':\text{ }}

\long\def\umbruch{\displaybreak[0]}
\def\equad{\mathrel{\phantom{=}}{}}
\long\def\neueZeile{{\rule{0mm}{1mm}\\[-3.25ex]\rule{0mm}{1mm}}}

\begin{document}

\title[Existence of solutions to a singular ODEs]{Existence proofs for
  rotationally symmetric translating solutions to mean curvature flow}

\author{Hakar Raji}

\author{Oliver C. Schn\"urer}
\address{Fachbereich Mathematik und Statistik, Universit\"at Konstanz,
  78457 Konstanz, Germany}
\def\AmSeeHome{@uni-konstanz.de}
\email{Hakar.Raji\AmSeeHome}
\email{Oliver.Schnuerer\AmSeeHome}
\thanks{}



\date{\today}

\dedicatory{}

\begin{abstract}
  There exist rotationally symmetric translating solutions to mean
  curvature flow that can be written as a graph over Euclidean
  space. This result is well-known. Its proof uses the symmetry and
  techniques from partial differential equations. However, the result
  can also be formulated as an existence result for a singular
  ordinary differential equation. Here, we provide different methods
  to prove existence of these solutions based on the study of the
  singular ordinary differential equation without using methods from
  partial differential equations.
\end{abstract}

\maketitle

\setcounter{tocdepth}{1}
\tableofcontents

\section{Introduction}

Translating solutions appear as a type II blow up in mean curvature
flow. If they are graphical, they can be written as $\graph u$ for
some function $u\colon\Omega\to\R$ and some domain
$\Omega\subset\R^n$. Then $u$ fulfils the partial differential
equation
\begin{equation}
  \label{pde}
  1=\sqrt{1+|\nabla u|^2}\cdot\text{div}\left(\frac{\nabla
      u}{\sqrt{1+|\nabla u|^2}}\right)
\end{equation}
in $\Omega$. For $n=1$, the grim reaper is an explicitly known
solution. For $n\ge2$, particularly simple examples are rotationally
symmetric solutions. S.~Altschuler and L.~Wu proved existence
\cite{AltschulerWu} for $n=2$ by solving the partial differential
equation on a sequence of balls $B_r(0)$ with constant boundary
values. As these solutions are rotationally symmetric, those defined
on different balls coincide up to an additive constant. Choosing the
additive constants such that $u(0)$ is independent of $r$, they find a
solution defined on all of $\R^n$. Their proof extends directly to
higher dimensions, see \cite{JCOSFSMCFStability}.

During the last years, there has been some interest in translating
solutions to mean curvature flow. For example, D. Hoffman, F. Mart\'in
and B. White have found Sherk-like translators
\cite{ScherklikeTranslators}.

Here, we address existence of rotationally symmetric solutions to
\eqref{pde} differently. We assume rotational symmetry $u(x)=U(|x|)$
and deduce in Lemma \ref{ode deriv lem} the ordinary differential
equation
\begin{equation}
  \label{phi intro eq}
  \phi'(r) =\left(1+\phi^2(r)\right)\left(1-\frac{n-1}r\phi(r)\right)
\end{equation}
for $\phi=U'$ and $r>0$. Our main result is
\begin{theorem}
  Let $n\ge2$. Then there exists a function
  $\phi\colon[0,\infty)\to\R$ solving \eqref{phi intro eq} such that
  \[\R^n\ni x\mapsto u(x)=\int\limits_0^{|x|}\phi(r)\,dr\] is of class
  $C^2$ and solves \eqref{pde}.
\end{theorem}
Here, the function $u$ is defined precisely in such a way to reverse
the construction above that lead from $u$ to $\phi$. In order to prove
this result, we establish the existence of a solution $\phi$ to
\eqref{phi intro eq} that fulfils conditions that ensure
$C^2$-regularity of $u$ at the origin. By continuity, $u$ solves
\eqref{pde} not only in $\R^n\setminus\{0\}$, but in all of $\R^n$.

The motivation to write this paper is to provide an existence proof
for rotationally symmetric solutions to mean curvature flow that is
accessible for people that do not (yet) know much about partial
differential equations. As a side effect, we encounter several
techniques that are useful in nonlinear analysis.

The rest of the paper is organised as follows. In Section \ref{ode
  deriv sec}, we derive the ordinary differential equation (ODE) for
$\phi$, discuss regularity at the origin and prove existence for
$r\ge r_0>0$.

Each of the subsequent five sections addresses a method to solve
\eqref{phi intro eq} and to establish $C^2$-regularity at the origin.

In Section \ref{weighted spaces sec}, we employ weighted
spaces with norm $\sup_{r>0}r^{-p}|u(r)|$ and adapt the standard
existence proof of Picard-Lindel\"of to those norms. For this method,
we also need to find sufficiently good approximative solutions. 

In Section \ref{shooting technique sec}, we solve the ODE with initial
value at some positive $r$. If the initial value is too big, solutions
will tend to infinity as $r\downarrow0$, and if it is too small, they
will tend to $-\infty$. By continuous dependence on initial data, we
find the right initial value somewhere in the middle. Technically,
this is carried out by transforming the equation such that $r=0$ is
mapped to $\infty$. Then we inductively restrict possible initial
values in such a way that solutions stay nicely bounded for longer and
longer intervals. In the limit, we obtain a solution. We also
establish uniqueness.

In the following section, Section \ref{approx pbls sec}, we impose the
initial value $\phi\left(\frac1k\right)=\frac1n\frac1k$, which is
motivated by a formal expansion near $r=0$, and solve \eqref{phi intro
  eq} for $r\ge\frac1k$. Letting $k\to\infty$, we find a solution.

In Section \ref{reg eq sec}, we regularise \eqref{phi intro eq} using
a parameter $\epsilon>0$ such that the equation becomes regular near
$r=0$. Then we solve this equation subject to $\phi(0)=0$. As
$\epsilon\downarrow0$, the corresponding solutions converge to a
solution of \eqref{phi intro eq}.

A completely different approach is used in Section \ref{power series
  sec}. As solutions to analytic elliptic partial differential
equations are analytic \cite{MorreyMultiple}, we employ power
series. Assuming that $\phi$ can locally be written as a power series,
we derive a recursive formula for the coefficients in that power
series. Then we establish a decay property for the coefficients that
allows us to show that the convergence radius of the power series is
positive. Finally, we show that the power series solves \eqref{phi
  intro eq}.

\section{Derivation of the ordinary differential equation}
\label{ode deriv sec}

This section is mainly concerned with the derivation of the ordinary
differential equation and a discussion about $C^2$-regularity at the
origin of rotationally symmetric functions, but includes also a
paragraph about our notation and an existence proof for $r\ge r_0>0$. 

\subsection{Notation}
Here and in the following, we use $u_i=\frac{\partial u}{\partial
  x^i}$, \ldots{} and employ the Einstein summation convention to sum
over repeated upper and lower latin indices from $1$ to $n$. 

\subsection{Derivation of the equation}
Outside of the origin, Equation \eqref{pde} for rotationally symmetric
functions $u$ can easily be rewritten as an ordinary differential
equation. 
\begin{lemma}
  \label{ode deriv lem}
  Let $n\in\N$ and $0<R\le\infty$.  Consider a function
  $u\colon B_R(0)\setminus\{0\}\to\R$, $B_R(0)\subset\R^n$, of the form
  $u(x)=U(|x|)$. Then $u$ is a $C^2$-solution to \eqref{pde} in
  $B_R(0)\setminus\{0\}$ if and only if $U\in C^2((0,R))$ and 
  $\phi(r):=U'(r)$ solves
  \begin{equation}
    \label{ode}
    \phi'(r)=\left(1+\phi^2(r)\right)\left(1-\frac{n-1}r\phi(r)\right)
  \end{equation}
  for $r\in(0,R)$.
\end{lemma}
\begin{proof}
  \neueZeile
  \begin{itemize}
  \item[\Hinrichtung] Assume first that $u\in C^2$ solves \eqref{pde} in
    $B_R(0)\setminus\{0\}$. The function
    $(0,R)\ni r\mapsto u(re_1)=U(r)$ is of class $C^2$. Hence, we can
    calculate for $r=|x|>0$
    \begin{align*}
      u_i(x)
      &=U'(|x|)\frac{x_i}{|x|},\umbruch\\
      u_{ij}(x)
      &=U''(|x|)\frac{x_ix_j}{|x|^2}
        +U'(|x|)\frac1{|x|}\left(\delta_{ij}
        -\frac{x_ix_j}{|x|^2}\right),\umbruch\\
      |\nabla u|^2
      &=\left(U'(r)\right)^2\delta^{ij}\frac{x_ix_j}{|x|^2}\umbruch\\
      &=\left(U'(r)\right)^2,\umbruch\\
      1
      &=\sqrt{1+|\nabla u|^2} \cdot\text{div}\, \left(\dfrac{\nabla
        u}{\sqrt{1+|\nabla u|^2}}\right)\umbruch\\
      &=\delta^{ij} u_{ij}(x)-\frac{u_{ij}u^iu^j}{1+|\nabla
        u|^2}(x)\umbruch\\ 
      &=U''(r)\frac{\delta^{ij}x_ix_j}{|x|^2} +U'(r)\frac1{|x|}
        \delta^{ij}\left(\delta_{ij}-\frac{x_ix_j}{|x|^2}\right)\\
      &\equad-U''(r)\frac{x_ix_j}{|x|^2}\left(U'(r)\right)^2
        \frac{x^ix^j}{|x|^2}\frac1{1+\left(U'(r)\right)^2}\\
      &\equad-U'(r)\frac1{|x|}
        \left(\delta_{ij}-\frac{x_ix_j}{|x|^2}\right) \left(U'(r)\right)^2
        \frac{x^ix^j}{|x|^2} \frac1{1+\left(U'(r)\right)^2}
        \umbruch\\
      &=U''(r)\left(1-\frac{\left(U'(r)\right)^2}
        {1+\left(U'(r)\right)^2}\right)+\frac{n-1}rU'(r)-0\umbruch\\
      &=\frac{U''(r)}{1+\left(U'(r)\right)^2}
        +\frac{n-1}rU'(r)\umbruch\\
      &=\frac{\phi'(r)}{1+\phi^2(r)} +\frac{n-1}r\phi(r).
    \end{align*}
    Rearranging yields \eqref{ode}.
  \item[\Rueckrichtung] If $U\in C^2$ outside the origin, we see from
    calculations as above that $u\in C^2$ outside the origin and $u$
    fulfils \eqref{pde}.\qedhere
  \end{itemize}
\end{proof}

\begin{remark}
  Equation \eqref{ode} is special in the case $n=1$. An even solution
  is explicitly known, the grim reaper, $u(r)=\log\cos(r)$,
  $\phi(r)=u'(r)$, $-\frac\pi2<r<\frac\pi2$. Any solution is locally
  of the form $u(r-c_1)+c_2$, $c_i\in\R$. This answers the existence
  and uniqueness of solutions completely for $n=1$. Therefore, we may
  focus on the case $n\ge2$ in the rest of the paper. \par
  Observe that the grim reaper is only defined on a bounded interval,
  whereas maximal solutions to \eqref{ode} are defined for all $r>0$ as
  we will see later.
\end{remark}

In order to investigate regularity at the origin, 
we start with a result that allows to exclude points
of non-differentiability. It is also valid for functions defined on
$(-\epsilon,\epsilon)$, $\epsilon>0$. 
\begin{proposition}
  \label{remove pt diff prop}
  Let $u,v\colon(-1,1)\to\R$ be continuous. Assume that \[u\in
  C^1((-1,0)\cup(0,1))\] and $u'(x)=v(x)$ for all
  $x\in(-1,0)\cup(0,1)$. Then $u\in C^1((-1,1))$ and $u'=v$.
\end{proposition}
\begin{proof}
  Define $U\colon(-1,1)\to\R$ by \[U(x):= u(0) +\int\limits_0^x
    v(t)\,dt.\] We obtain that $U\in C^1((-1,1))$, $U'(x)=v(x)$ for
  all $x\in(-1,1)$ and therefore $U'(x)=u'(x)$ for all
  $x\in(-1,0)\cup(0,1)$. Hence, there exist constants $c_\pm$ such
  that $U(x)=u(x)+c_+$ for all $x\in(0,1)$ and $U(x)=u(x)+c_-$ for all
  $x\in(-1,0)$. We wish to show that $c_+=0=c_-$. For $x>0$, we
  obtain \[c_+ =U(x)-u(x) =u(0)+\int\limits_0^x v(t)\,dt -u(x)\] and
  deduce for $0<x\le\frac12$ \[|c_+|\le |u(0)-u(x)|
    +|x|\cdot\sup\limits_{\left[0,\frac12\right]} |v|.\] As $u$ and
  $v$ are continuous, the right-hand side of this inequality converges
  to zero as $x\downarrow0$ and we obtain that $c_+=0$. Similarly, we
  deduce that $c_-=0$. According to our definition, we see that
  $U(0)=u(0)$. Hence, we have shown that $U=u$ on $(-1,1)$. Therefore,
  our claims follow from the respective properties of $U$. 
\end{proof}

The following proposition characterises the $C^2$-regularity of
rotationally symmetric functions at the origin. 
\begin{proposition}
  \label{reg orig prop}
  Let $n\in\N$ with $n\ge2$ and $R>0$. 
  Let \[\Phi\in C^2((0,R))\cap C^0([0,R)).\] Then
  \begin{align*}
    u\colon B_R(0)\to&\,\R,\\
    u(x):=&\,\Phi(|x|)
  \end{align*}
  is of class $C^2$ if and only if the
  limits \[\lim\limits_{r\downarrow0}\frac{\Phi'(r)}r
    \quad\text{and}\quad \lim\limits_{r\downarrow0}\Phi''(r)\] exist
  and coincide. 
\end{proposition}
\begin{proof}
  \neueZeile
  \begin{itemize}
  \item[\Hinrichtung] Let $u\in C^2$ and $x\neq0$. According
    to our calculations in Lemma \ref{ode deriv lem}, the second
    derivatives of $u$ are given by
    \[u_{ij}(x)=\Phi''(|x|)\frac{x_i}{|x|}\frac{x_j}{|x|}
      +\Phi'(|x|)\frac1{|x|} \left(\delta_{ij}
        -\frac{x_ix_j}{|x|^2}\right).\]
    In particular, we obtain for $r\in(0,R)$
    \[u_{22}(re_1) =0+\Phi'(r)\frac1r(1-0) =\Phi'(r)\frac1r.\] By
    assumption, the limit of $u_{22}(re_1)$ as $r\downarrow0$
    exists. This implies that
    $\lim\limits_{r\downarrow0}\Phi'(r)\frac1r$ exists. As a
    consequence, we see that
    \[\lim\limits_{r\downarrow0}\Phi'(r)
      =\lim\limits_{r\downarrow0} \left(r\cdot\frac{\Phi'(r)}r\right)
      =0\cdot\lim\limits_{r\downarrow0}\frac{\Phi'(r)}r =0.\] \par We
    now define $\tilde\Phi(r):= u(re_1)$ and see that $\tilde\Phi$ is
    an extension of $\Phi$ to $(-R,R)$ and of class $C^2$. In
    particular, we get $\tilde\Phi'(0)=0$. According
    to the definition of $\tilde\Phi''$, we deduce that
    \begin{align*}
      \lim\limits_{r\downarrow0}\Phi''(r)
      &=\lim\limits_{r\to0}\tilde\Phi''(r)
      =\tilde\Phi''(0) \umbruch\\
      &=\lim\limits_{\genfrac{}{}{0pt}{}{r\to0}{r\neq0}}
        \frac{\tilde\Phi'(r)-\tilde\Phi'(0)}{r-0}
        =\lim\limits_{r\downarrow0}\frac{\Phi'(r)}r.
    \end{align*}
    Therefore,
    $\lim\limits_{r\downarrow0}\Phi''(r)$ exists and coincides with
    $\lim\limits_{r\downarrow0}\frac{\Phi'(r)}r$.\par
  \item[\Rueckrichtung]
    \begin{itemize}
    \item $u\in C^0$: As $\Phi$ is continuous on $[0,R)$, $u$ is
      continuous at the origin.
    \item $u\in C^1$: According our expressions for first and second
      derivatives in Lemma \ref{ode deriv lem}, it is clear that $u$
      is of class $C^2$ outside the origin. As the limit
      $\lim\limits_{r\downarrow0}\frac{\Phi'(r)}r$ exists, we conclude
      as above that $\lim\limits_{r\downarrow0}\Phi'(r)=0$.
      Therefore, we obtain that
      \[\lim\limits_{\genfrac{}{}{0pt}{}{x\to0}{x\neq0}}u_i(x)
        =\lim\limits_{\genfrac{}{}{0pt}{}{x\to0}{x\neq0}}\Phi'(|x|)
        \frac{x_i}{|x|}=0.\] Now, we apply Proposition \ref{remove pt
        diff prop} on a small interval around the origin, to the
      function $t\mapsto u(te_i)$, $1\le i\le n$, and deduce that it
      is differentiable at $t=0$ and $u_i(0)=0$. This implies that $u$
      is partially differentiable in $B_R(0)$ and the partial
      derivatives are continuous. Hence, we obtain that
      $u\in C^1(B_R(0))$.
    \item $u\in C^2$: We rewrite $u_{ij}(x)$, $x\neq0$, as follows
      \begin{align*}
        u_{ij}(x)=&\,\Phi''(|x|)\frac{x_i}{|x|}\frac{x_j}{|x|}
                    +\Phi'(|x|)\frac1{|x|} \left(\delta_{ij}
                    -\frac{x_i}{|x|}\frac{x_j}{|x|} \right)\umbruch\\ 
        =&\,\left(\Phi''(|x|)-\Phi'(|x|)\frac1{|x|}\right)
           \frac{x_i}{|x|}\frac{x_j}{|x|}
           +\Phi'(|x|)\frac1{|x|}\delta_{ij}
      \end{align*}
      and deduce that
      $\lim\limits_{\genfrac{}{}{0pt}{}{x\to0}{x\neq0}} u_{ij}(x)$
      exists and 
      equals $\lim\limits_{r\downarrow0} \Phi'(r)\frac1r
      \delta_{ij}$. We can now apply the arguments used in the proof
      that $u\in C^1$ to the partial derivatives $u_j$, $1\le j\le n$,
      and deduce that $u\in C^2$. \qedhere
    \end{itemize}
  \end{itemize}
\end{proof}

For later use, it will be convenient to reformulate the regularity
conditions at the origin. 
\begin{corollary}
  \label{reg orig exp cor}
  Let $n$, $R$, $\Phi$ and $U$ be as in Proposition \ref{reg orig
    prop}.
  Let $\psi(r):=\Phi'(e^{-r})$. Then $U$ is of class $C^2$ if
  and only if the limits
  \[\lim\limits_{r\to\infty} e^r\psi(r) \quad\text{and}\quad
    \lim\limits_{r\to\infty} e^r\psi'(r)\] exist and sum
  up to zero.
\end{corollary}
\begin{proof}
  According to the definition of $\psi$, we obtain
    \[\psi'(r) =-\Phi''(e^{-r})e^{-r}.\] Hence, we can rewrite the
    limits as
    \begin{align*}
      \lim\limits_{r\downarrow0}\frac{\Phi'(r)}r
      =&\,\lim\limits_{r\to\infty}\frac{\Phi'(e^{-r})}{e^{-r}}
         =\lim\limits_{r\to\infty} e^r\psi(r)\umbruch
         \intertext{and}
         \lim\limits_{r\downarrow0}\Phi''(r)
         =&\,\lim\limits_{r\to\infty} \Phi''(e^{-r})
            =\lim\limits_{r\to\infty}-e^r\psi'(r). 
    \end{align*}
    Therefore, our claim follows from Proposition \ref{reg orig prop}. 
\end{proof}

\subsection{Existence away from from the origin}

Once existence near the origin is established, it is straight-forward
to prove global existence. 
\begin{lemma}
  \label{existence large r lem}
  Let $n\in\N$ with $n\ge2$, $r_0>0$ and $a\in\R$. Then there exists a
  smooth solution $\phi\colon[r_0,\infty)\to\R$ to \eqref{phi intro
    eq} with $\phi(r_0)=a$.
\end{lemma}
\begin{proof}
  The Picard-Lindel\"of existence theorem ensures that $\phi$ exists
  on some interval $[r_0,r_0+\epsilon)$, $\epsilon>0$. By the
  characterisation of the maximal existence interval, it suffices to
  show that $|\phi|$ cannot tend to infinity on any finite
  interval. Using the ODE, we see that $\phi(r)<0$ implies
  $\phi'(r)>0$ and $\phi(r)>\frac r{n-1}$ implies
  $\phi'(r)<0$. Therefore, we deduce that 
  \[\min\{0,a\}\le\phi(r)\le\max\left\{\frac r{n-1},a\right\}.\]
  This guarantees longtime existence.
\end{proof}

This Lemma allows to extend a solution defined in a neighbourhood of
the origin: Assume that $\phi$ is defined on $[0,2r_0)$ for some
$r_0>0$. Then, according to Lemma~\ref{existence large r lem}, there
exists a solution $\hat\phi$ on $[r_0,\infty)$ with
$\hat\phi(r_0) =\phi(r_0)$. By uniqueness due to Picard-Lindel\"of,
$\phi$ and $\hat\phi$ coincide on $[r_0,r_0+\epsilon)$ for some
$\epsilon\in(0,r_0)$. Thus $\phi$ can be extended smoothly to a
solution of \eqref{phi intro eq} on $[0,\infty)$.

\section{Existence via weighted spaces}
\label{weighted spaces sec}
We give an existence proof for \eqref{phi intro eq} near $r=0$
based on weighted spaces.
\begin{lemma}
  \label{C0p Banach space lem}
  Let $r_0>0$ and $p>0$. We define for $u\in C^0([0,r_0])$
  \[\Vert u\Vert_{C^0_{-p}([0,r_0])} \equiv \Vert
    u\Vert_{C^0_{-p}}:=\sup\limits_{r_0\ge r>0}
    r^{-p}|u(r)|\in[0,\infty]\] and the space
  \[C^0_{-p}([0,r_0]) :=\left\{u\in C^0([0,r_0])\colon \Vert
      u\Vert_{C^0_{-p}}<\infty\right\}.\] Then
  $\Vert\cdot\Vert_{C^0_{-p}}$ is a norm on $C^0_{-p}([0,r_0])$ and
  $C^0_{-p}([0,r_0])$ equipped with the norm
  $\Vert\cdot\Vert_{C^0_{-p}}$ is a Banach space.
\end{lemma}
Henceforth, we will always use the norm $\Vert\cdot\Vert_{C^0_{-p}}$
on $C^0_{-p}([0,r_0])$.
\begin{proof}
  \neueZeile
  \begin{enumerate}[label=(\roman*)]
  \item \emph{Norm:} Functions $u\in C^0_{-p}([0,r_0])$ fulfil
    $|u(r)|\le C\cdot r^p$ for all $r\in[0,r_0]$ for some $C\ge0$. In
    particular, we obtain that $u(0)=0$ and that
    $\Vert\cdot\Vert_{C^0_{-p}}$ is positive definite. It is now easy
    to see that $\Vert\cdot\Vert_{C^0_{-p}}$ is a norm.
  \item \emph{Banach space:} Let
    $(u_n)_{n\in\N}\subset C^0_{-p}([0,r_0])$ be a Cauchy sequence. We
    wish to show that it converges to some element in
    $C^0_{-p}([0,r_0])$. \par Let $\epsilon\in(0,r_0]$. Then the norm
    \[\Vert u\Vert_{C^0_{-p}([\epsilon,r_0])}
      :=\sup\limits_{r\in[\epsilon,r_0]} r^{-p}|u(r)|\] on
    $C^0([\epsilon,r_0])$ is equivalent to the
    $\Vert\cdot\Vert_{C^0}$-norm, namely
    \[r_0^{-p}\cdot\Vert u\Vert_{C^0} \le\Vert
      u\Vert_{C^0_{-p}([\epsilon,r_0])}
      =\sup\limits_{r\in[\epsilon,r_0]} r^{-p}|u(r)|\le
      \epsilon^{-p}\cdot\Vert u\Vert_{C^0}\] for all
    $u\in C^0([\epsilon,r_0])$. For any $\epsilon\in(0,r_0]$, the
    sequence $(u_n|_{[\epsilon,r_0]})_{n\in\N}$ is a Cauchy sequence
    with respect to $\Vert\cdot\Vert_{C^0_{-p}([\epsilon,r_0])}$ and
    hence also in $C^0([\epsilon,r_0])$. Since $C^0$ is a Banach
    space, the sequence $\left(u_n|_{[\epsilon,r_0]}\right)_{n\in\N}$
    has a continuous limit $u_\epsilon\colon[\epsilon,r_0]\to\R$ with
    $u_n|_{[\epsilon,r_0]}\to u_\epsilon$ uniformly as
    $n\to\infty$. It is easy to see that
    $u_{\epsilon_1}|_{[\epsilon_2,r_0]} = u_{\epsilon_2}$ for any
    $0<\epsilon_1<\epsilon_2<r_0$. Therefore, we can define \[u(r):=
      \begin{cases}
        u_\epsilon(r), &r\in (0,r_0] \text{ for some }0<\epsilon\le
        r,\\ 
        0, &r=0. 
      \end{cases}\] As $(u_n)_{n\in\N}$ is a Cauchy sequence and hence
    bounded, there exists $C\ge0$ such that $|u_n(r)|\le C\cdot r^p$
    for all $r\in[0,r_0]$ and all $n\in\N$. The pointwise convergence
    $u_n\to u$ as $n\to\infty$ ensures that $|u(r)|\le C\cdot r^p$ for
    all $r\in[0,r_0]$.  Therefore, the function $u$ is continuous at
    the origin. The continuity of $u$ at $r>0$ is a direct consequence
    of the uniform convergence. Thus, we see that $u\in C^0_{-p}$.
    Finally, letting $m\to\infty$ in the Cauchy condition
    \[\forall\,\epsilon>0\,\,\exists\,N\in\N\,\,\forall\,n,m\ge
      N\,\,\forall\,r\in[0,r_0]: |u_n(r)-u_m(r)|\le\epsilon\cdot r^p\]
    yields the same condition with
    $|u_n(r)-u(r)|\le\epsilon\cdot r^p$. Therefore, we deduce that
    $u_n\to u$ in $\Vert\cdot\Vert_{C^0_{-p}}$ and conclude that
    $C^0_{-p}([0,r_0])$ is a Banach space. \qedhere
  \end{enumerate}
\end{proof}

\begin{remark}
  Lemma \ref{C0p Banach space lem} extends to the situation
  $u\colon[0,r_0]\to E$ for a Banach space~$E$.
\end{remark}

We will need two simple lemmata that imply in different settings that
$r\mapsto\frac{u(r)}r$ is continuous.
\begin{lemma}
  \label{ur cont lem}
  Let $p>1$ and $u\in C^0_{-p}([0,T])$, $T>0$. Then the function
  \[r\mapsto
    \begin{cases}
      \frac{u(r)}r,& 0<r\le T,\\
      0,& r=0,
    \end{cases}\] denoted by $r\mapsto\frac{u(r)}r$, $r\in[0,T]$, in
  the following, is continuous,
  \[\left(r\mapsto\frac{u(r)}r\right)\in C^0([0,T]).\]
\end{lemma}
\begin{proof}
  Consider \[\left|\frac{u(r)}r\right|\le r^{p-1} \cdot \Vert
    u\Vert_{C^0_{-p}([0,T])}\] and use that $p-1>0$. 
\end{proof}

\begin{lemma}
  \label{ur cont lem 2}
  Let $T>0$ and $u\in C^1([0,T])$ with $u(0)=0$. Then the function
  \[r\mapsto
    \begin{cases}
      \frac{u(r)}r,&0<r\le T,\\
      u'(0),&r=0,      
    \end{cases}
  \]
  denoted by $r\mapsto\frac{u(r)}r$, $r\in[0,T]$, in the following, is
  continuous, \[\left(r\mapsto\frac{u(r)}r\right)\in C^0([0,T]).\]
\end{lemma}
\begin{proof}
  It is tempting to use
  \[\frac{u(r)}r =\frac{u(r)-u(0)}{r-0}
    \underset{r\downarrow0}\longrightarrow u'(0),\] but, according to
  the definition of $C^1([0,T])$, $u$ is only differentiable in
  $(0,T)$ and the derivative can be extended continuously to
  $[0,T]$. Therefore, we will avoid using the difference quotient at
  the origin. \par Let $r>0$ and $0<s<r$. Then the fundamental theorem
  of calculus implies that
  \[\frac{u(r)}r\underset{s\downarrow0}{\longleftarrow}
    \frac{u(r)-u(s)}{r-s} =\frac1{r-s}\int\limits_s^r u'(\tau)\,d\tau
    \underset{s\downarrow0}{\longrightarrow} \frac1r\int\limits_0^r
    u'(\tau)\,d\tau.\] As limits are unique, we infer
  that \[\frac{u(r)}r =\frac1r\int\limits_0^r u'(\tau)\,d\tau.\] The
  right-hand side is continuous up to $r=0$ with limit $u'(0)$ at
  $r=0$ as claimed.
\end{proof}

The following theorem is a variant of the Picard-Lindel\"of Theorem
for differential equations of the form $\dot u(r)
=f\left(r,\frac{u(r)}r\right)$. Note that this includes right-hand
sides $f$ that depend on $u(r)=r\cdot\frac{u(r)}r$. Furthermore, we
wish to point out that in general $h,u\not\in C^0_{-p}([0,S])$. 

\begin{theorem}
  \label{1r sing ex thm}
  Let $N\in\N_{\ge1}$, $S\in(0,1]$, $R>0$ and $L\ge0$. Let
  $f\colon[0,S]\times \overline{B_{2R}(0)}\to\R^N$ with
  $B_{2R}(0)\subset\R^N$ be continuous and uniformly Lipschitz
  continuous with respect to the second argument, more precisely,
  assume that $|f(s,a)-f(s,b)| \le L\cdot|a-b|$ for all $s\in[0,S]$
  and all $a,b\in\overline{B_{2R}(0)}$. Let $p>\max\{L,1\}$ and let
  $h\in C^1\left([0,S],\R^N\right)$ be a function such that
  $\frac{h(s)}s\in\overline{B_R(0)}$ for all $s\in(0,S]$ and hence in
  particular $h(s)\in\overline{B_R(0)}$ for all $s\in[0,S]$ and $h(0)=0$.
  Assume that $h$ is an approximate solution of
  $u'(r)=f\left(r,\frac{u(r)}r\right)$ in the sense that
  \[\left\Vert r\mapsto\left( h(r)-\int\limits_0^r
        f\left(s,\frac{h(s)}s\right)\,ds\right)
    \right\Vert_{C^0_{-p}} \le \frac{p-L}p\cdot R. \] Then the
  differential equation
  \[u'(r) =f\left(r,\frac{u(r)}r\right),\quad r\in(0,S],\] has a
  unique solution $u$ such that $\Vert u-h\Vert_{C^0_{-p}}\le R$. \par
  Moreover, $u\in C^1\left([0,S],\R^N\right)$ and $u(0)=0$.
\end{theorem}
\begin{proof}
  \neueZeile
  \begin{enumerate}[label=(\roman*)]
  \item We proceed similarly as in the proof of the Picard-Lindel\"of
    short time existence result and consider the integral equation
    \[u(r) =\int\limits_0^r f\left(s, \frac{u(s)}s\right)\,ds.\]
    In contrast to the situation of the Picard-Lindel\"of Theorem, it
    is not obvious that this integral equation is equivalent to the
    differential equation. Here, our strategy is to solve the integral
    equation, establish additional regularity of the solution and
    finally infer that $u$ also solves the differential equation.
  \item We define
    \[\mathcal M:= \overline{B_R(0)}\subset
      C^0_{-p}\left([0,S],\R^N\right)\] as the 
    closed ball around the origin with respect to the
    $C^0_{-p}$-norm. We also define the operator
    \[T\colon \mathcal M \to C^0\left([0,S],\R^N\right)\] for
    $r\in[0,S]$ by
    \[(T(\phi))(r) := \int\limits_0^r
      f\left(s,\frac{(h+\phi)(s)}s\right)\,ds -h(r).\] We wish to
    point out that it is tempting to remove $h$ in the definition of
    $T$. But we have to shift the situation by $h$ as
    $h+\phi\not\in C^0_{-p}\left([0,S],\R^N\right)$ in general. \par
    Our assumptions $\frac{h(s)}s\in\overline{B_R(0)}$ for all $s\in(0,S]$,
    $\phi\in C^0_{-p}$, $p>1$ and $S\in(0,1]$ ensure that the
    integrand is defined. \par
    Lemmata \ref{ur cont lem} and \ref{ur cont lem 2} ensure that
    the integrand is continuous. Hence $T$ maps indeed to
    $C^0\left([0,S],\R^N\right)$, and, by the fundamental theorem of
    calculus, even to $C^1\left([0,S],\R^N\right)$. Therefore, every
    fixed point $\phi$ of $T$ and therefore also a prospective
    solution $u=h+\phi$ have the claimed regularity.
  \item As in the Picard-Lindel\"of Theorem, we wish to apply the
    Banach fixed point theorem. In order to do this, we have to show
    that $T$ is a contraction and that $T(\mathcal M)\subset\mathcal
    M$. 
  \item\label{1r ex thm cont prop} \emph{Contraction:} Let
    $\phi,\psi\in \mathcal M$. We wish to show that there exists some
    constant $c<1$ such that
    \begin{align*}
      \Vert T(\phi)-T(\psi)\Vert_{C^0_{-p}}
      &\le c\cdot \Vert\phi-\psi\Vert_{C^0_{-p}}.
    \end{align*}
    To achieve this, we consider arbitrary $r\in(0,S]$ and obtain
    \begin{align*}
      &\equad r^{-p}\cdot |(T(\phi))(r)-(T(\psi))(r)|\umbruch\\
      &=\frac1{r^p}\cdot\left|\int\limits_0^r
        f\left(s,\frac{(h+\phi)(s)}s\right)\,ds -\int\limits_0^r
        f\left(s,\frac{(h+\psi)(s)}s\right)\,ds\right|\umbruch\\
      &\le\frac1{r^p}\cdot \int\limits_0^r
        \left|f\left(s,\frac{(h+\phi)(s)}s\right)
        -f\left(s,\frac{(h+\psi)(s)}s\right)\right|\,ds\umbruch\\
      &\le\frac1{r^p}\cdot\int\limits_0^r
        L\cdot\frac1s\cdot|(h+\phi)(s) -(h+\psi)(s)|\,ds\umbruch\\
      &=\frac1{r^p}\cdot L\cdot\int\limits_0^r
        s^{p-1}\cdot\frac{|\phi(s)-\psi(s)|}{s^p}\,ds\umbruch\\
      &\le\frac1{r^p}\cdot L\cdot\int\limits_0^r s^{p-1}\,ds\cdot
        \Vert\phi-\psi\Vert_{C^0_{-p}}\umbruch\\
      &=\frac1{r^p}\cdot L\cdot\frac1p\cdot
        r^p\cdot\Vert\phi-\psi\Vert_{C^0_{-p}}\umbruch\\
      &=\frac Lp\cdot\Vert\phi-\psi\Vert_{C^0_{-p}}
    \end{align*}
    with $\frac Lp<1$ as claimed. 
  \item \emph{Self-map:} We consider $\phi\in \mathcal M$. Our aim is
    to verify that $T(\phi)\in\mathcal M$ or, equivalently,
    $\Vert T(\phi)\Vert_{C^0_{-p}}\le R$.  We consider an arbitrary
    $r\in(0,S]$ and estimate
    \begin{align*}
      &\equad r^{-p}\cdot|(T(\phi))(r)|\umbruch\\
      &=\frac1{r^p}\cdot\left| \int\limits_0^r
        f\left(s,\frac{(h+\phi)(s)}s\right)\,ds -h(r)\right|\umbruch\\
      &\le\frac1{r^p}\cdot \left|\int\limits_0^r
        f\left(s,\frac{(h+\phi)(s)}s\right)
        -f\left(s,\frac{h(s)}s\right)\,ds\right|\\
      &\equad\quad+\frac1{r^p}\cdot
        \left|\int\limits_0^r f\left(s,\frac{h(s)}s\right)\,ds
        -h(r)\right|\umbruch\\
      &\le\frac1{r^p}\cdot|(T(\phi))(r) -(T(0))(r)|
        +\frac{p-L}p\cdot R\umbruch\\
      &\le\frac Lp\cdot\Vert\phi-0\Vert_{C^0_{-p}}
        +\frac{p-L}p\cdot R\umbruch\\
      &\le\frac Lp\cdot R+\frac{p-L}p\cdot R =R. 
    \end{align*}
    Here, we have used the estimate from \ref{1r ex thm cont prop}
    and our assumption on the approximate solution $h$. This proves
    that $T(\mathcal M)\subset\mathcal M$.  
  \item Now, the Banach fixed point theorem yields a fixed point
    $T(\phi) =\phi\in\mathcal M$. As discussed above,
    $u:=h+\phi\in C^1\left([0,S],\R^N\right)$, $u(0)=0$ and
    $\left(r\mapsto\frac{u(r)}r\right) \in
    C^0\left([0,S],\R^N\right)$. Therefore, differentiating
    $(T(\phi))(r) =\phi(r)$ with respect to $r$ shows that
    \[u'(r) =h'(r)+\phi'(r) =\frac{d}{dr} \int\limits_0^r
      f\left(s,\frac{(h+\phi)(s)}{s}\right)\,ds
      =f\left(r,\frac{u(r)}r\right)\] for all $r\in[0,S]$ as
    claimed.
  \item \emph{Uniqueness:} Let $u\in C^1$ with
    $\Vert u-h\Vert_{C^0_{-p}}\le R$ be a solution to
    $u'(r)=f\left(r,\frac{u(r)}r\right)$. Then we define $\phi:=u-h$
    and see that $r\mapsto\frac{\phi(r)}r$ and $r\mapsto\frac{h(r)}r$
    are continuous. Integrating the differential equation yields
    \[\phi(r)+h(r) =u(r) =u(r)-u(0) =\int\limits_0^r
      f\left(s,\frac{\phi(s)+h(s)}s\right)\,ds.\]
    Thus $\phi$ is is fixed point for $T$ with
    $\Vert\phi\Vert_{C^0_{-p}}\le R$ and uniqueness follows from the
    Banach fixed point theorem. \qedhere
  \end{enumerate}
\end{proof}

We wish to apply Theorem \ref{1r sing ex thm} to solve \eqref{phi
  intro eq}. This requires approximate solutions. We will first
illustrate how to find those approximate solutions. Then we will prove
that arbitrarily good approximations do indeed exist.
\begin{remark}
  \label{expl asymp rem}
  Assume that $n\ge2$.  We wish to find a power series approximation
  such that \eqref{phi intro eq} is fulfilled for $r\ge0$ close to
  zero to some order. In order to do so, we set
  \begin{align*}
    \phi(r)
    &=a+br+cr^2+dr^3
      \intertext{and obtain}
      \phi'(r)
    &=b+2cr+3dr^2.
  \end{align*}
  Now, we want to choose $a,b,c,d\in\R$ successively such that the
  left-hand side minus the right-hand side of \eqref{phi intro eq},
  \[\phi'(r)-\left(1+\phi^2(r)\right)\cdot
    \left(1-\frac{n-1}r\phi(r)\right),\] 
  multiplied with $r^1$, then with $r^0$, $r^{-1}$, \ldots{} vanishes
  in the limit $r\downarrow0$. \par
  In the first step, we multiply with $r^1$ and obtain
  \begin{align*}
    0&\overset!=\lim\limits_{r\to0}
       r\cdot\left(\phi'(r)-\left(1+\phi^2(r)\right)
       \cdot\left(1-\frac{n-1}r\phi(r)\right)\right)\umbruch\\
     &=\lim\limits_{r\to0}
       \left\{br+2cr^2+3dr^3-\left(1+\phi^2(r)\right)
       \cdot\left(r-r\frac{n-1}r  
       \phi(r)\right)\right\}\umbruch\\
     &=\left(1+a^2\right)(n-1)a
  \end{align*}
  and deduce that $a=0$. \par Next, we multiply with $r^0=1$ and get
  \begin{align*}
    0&\overset!=b-\lim\limits_{r\to0}\left(1+b^2r^2\right)
       \cdot\left(1-\frac{n-1}r br\right)\umbruch\\
    &=b-(1-(n-1)b) =nb-1
  \end{align*}
  and see that $b=\frac1n$. Here, we have already dropped irrelevant
  higher order terms in the first line. Alternatively, it is possible
  to set $c=d=0$ in this step, to determine $b$ from the ansatz
  $\phi(r)=0+br$ and to find $c$ and $d$ later.
  \par Now, we multiply with $\frac1r$, calculate as above
  \begin{align*}
    0&\overset!=\lim\limits_{r\to0}\frac1r
       \left\{\frac1n+2cr-\left(1+\frac{r^2}{n^2}\right)
       \cdot\left(1-\frac{n-1}r\left(\frac1nr
       +cr^2\right)\right)\right\}\umbruch\\
    &=\lim\limits_{r\to0}\left\{\frac1{nr}+2c
      -\frac1r+\frac{n-1}{nr}+(n-1)c\right\}\umbruch\\
    &=(n+1)c
  \end{align*}
  and conclude that $c=0$. \par
  Finally, a similar calculation or the
  \texttt{python}-program using \texttt{Sympy} \cite{sympy}
\begin{verbatim}
from sympy import *
r, n, d = symbols('r n d')
phi = r / n + d*r**3
ode = diff(phi, r) - (1 + phi**2) * (1 - (n-1) / r * phi)
print(limit(ode / r**2, r, 0))
\end{verbatim}
  shows that $d=\frac1{n^3}\frac1{n+2}$. Recall that \texttt{r**3}
  corresponds to $r^3$. \par
  Following this approach further yields the expectation that
  \[\phi(r) \approx \frac{r}{n} +\frac{\left(\frac rn\right)^3}{n+2}
    -\frac{(n - 3)\cdot \left(\frac rn\right)^5}{n^2 + 6n + 8}
    +\frac{\left(n^3 - 6n^2 - 8n + 30\right)\cdot \left(\frac
        rn\right)^7}{n^4 + 14n^3 + 68n^2 + 136n + 96}\]
  for $r\ge0$ near $r=0$.
\end{remark}

As usual, we denote real polynomials in the variable $r$ by $\R[r]$
and define \[\R_0[r]:=\{\phi\in\R[r]\colon \phi(0)=0\}\] to be the
subspace of those polynomials that vanish at the origin.

Near the origin, we can follow this strategy and find approximate
polynomial solutions of arbitrary precision.
\begin{lemma}
  \label{pol approx diff eq lem}
  Let $n\in\N_{\ge2}$. We define the map $G\colon\R_0[r]\to\R[r]$ by
  \[(G(\phi))(r) :=\phi'(r) -\left(1+\phi^2(r)\right)
    \cdot\left(1-\frac{n-1}r\cdot \phi(r)\right).\] Let $k\in\N$. Then
  there exists a real polynomial function $\phi$ of degree at most $k$
  with $\phi(0)=0$ and \[(G(\phi))(r) \in O\left(r^k\right)\] for
  $r\ge0$ near the origin.
\end{lemma}
\begin{proof}
  We use induction on $k$.
  \begin{enumerate}[label=(\roman*)]
  \item For $k=0$, the only polynomial of degree $0$ with $\phi(0)=0$
    is given by $\phi(r)=0$ for all $r\in\R$. We get
    $(G(\phi))(r) =-1\in O\left(r^0\right)$ as claimed.
  \item Assume now that $k\ge1$ and that $\phi$ is a polynomial with
    $\deg\phi\le k$, $\phi(0)=0$ and
    $(G(\phi))(r)\in O\left(r^k\right)$. Then there exists some $b\in\R$
    such that
    \[\phi'(r) =\left(1+\phi^2(r)\right)\cdot
      \left(1-\frac{n-1}r\cdot\phi(r)\right) +br^k
      +O\left(r^{k+1}\right).\] We
    define \[\psi(r):=\phi(r)+ar^{k+1}\] and wish to show that we can
    choose $a$ such that $(G(\psi))(r) \in
    O\left(r^{k+1}\right)$. Using our assumption on $\phi$, we obtain
    \begin{align*}
      (G(\psi))(r)
      &=\phi'(r) +a(k+1)r^k\\
      &\equad -\left(1+\left(\phi(r)+ar^{k+1}\right)^2\right) \cdot
        \left(1-\frac{n-1}r\cdot
        \left(\phi(r)+ar^{k+1}\right)\right)\umbruch\\  
      &=\left(1+\phi^2(r)\right)\cdot
        \left(1-\frac{n-1}r\cdot\phi(r)\right) +br^k +a(k+1)r^k\\
      &\equad {}-\left(\left(1+\phi^2(r)\right) +\left(2a\phi(r)r^{k+1}
        +a^2r^{2k+2}\right)\right)\\
      &\equad\qquad\cdot
        \left(\left(1-\frac{n-1}r\cdot\phi(r)\right) -(n-1)ar^k\right)\\
      &\equad {}+O\left(r^{k+1}\right)\umbruch\\
      &=br^k + a(k+1)r^k\\
      &\equad{}+\left(1+\phi^2(r)\right)(n-1)ar^k\\
      &\equad{}-\left(2a\phi(r)r^{k+1} +a^2r^{2k+2}\right) \cdot
        \left(1-\frac{n-1}r\cdot\phi(r) -(n-1)ar^k\right)\\
      &\equad{}+O\left(r^{k+1}\right)\umbruch\\
      &=\left[b+a(k+1)+a(n-1)\right]\cdot r^k +O\left(r^{k+1}\right). 
    \end{align*}
    We choose $a=-\frac b{n+k}$ and see that $\psi$ is a polynomial as
    claimed. \qedhere
  \end{enumerate}
\end{proof}

Integrating yields
\begin{corollary}
  \label{approx poly fct cor}
  Let $n\in\N_{\ge2}$. Let $\epsilon>0$ and $p\in\N$ be given. Then
  there exists some real polynomial function $h$ with $h(0)=0$ and
  some $R,S>0$ such that
  \[\left\Vert r\mapsto\left(h(r)-\int\limits_0^r
        \left(1+h^2(s)\right)\cdot
        \left(1-\frac{n-1}s\cdot h(s)\right)\,ds\right)
    \right\Vert_{C^0_{-p}([0,S])}\le\epsilon,\]
  $\left|\frac{h(s)}s\right|\le R$ for all $s\in(0,S]$ 
  and $h\in C^1\left([0,S],\overline{B_R(0)}\right)$. 
\end{corollary}
\begin{proof}
  We apply Lemma \ref{pol approx diff eq lem} with $k=p$ and get a
  polynomial $h:=\phi$ such that
  \[f(r):= h'(r)-\left(1+h^2(r)\right)\cdot \left(1-\frac{n-1}r\cdot
      h(r)\right)\in O\left(r^p\right).\]
  We obtain $|f(r)|\le Cr^p$ for some $C\in\R$. We integrate this
  inequality and obtain
  \begin{align*}
    \frac1{r^p}\left|\int\limits_0^r f(s)\,ds\right|\le \frac1{r^p}
    \int\limits_0^r Cs^p\,ds =\frac1{r^p}\frac C{p+1} r^{p+1} =\frac
    C{p+1}r\le \frac C{p+1}S.  
  \end{align*}
  For $S\le\frac{\epsilon (p+1)}C$, we deduce the claimed estimate.
  We will assume in the following, that $S>0$ is fixed such that
  $S\le1$. 
  \par As $h$ is a polynomial with $h(0)=0$, choosing $R$ sufficiently
  large, yields $\left|\frac{h(s)}s\right|\le R$ for all
  $s\in(0,S]$. We wish to point out that, in view of the asymptotics
  $\phi(r)\approx\frac rn$ obtained in Remark \ref{expl asymp rem}, it
  is possible to use any $R>\frac1n$ if $S>0$ is chosen sufficiently
  small.  As $S\le1$, we deduce that
  $|h(s)|\le s\cdot R\le S\cdot R\le R$ for all $s\in[0,S]$.  Thus, we
  finally see that $h\in C^1\left([0,S],\overline{B_R(0)}\right)$.
\end{proof}

Finally, we apply the above results to show that \eqref{phi intro eq}
has a solution near $r=0$ that fulfils the regularity condition of
Proposition \ref{reg orig prop}.
\begin{theorem}
  There exists some $\epsilon>0$ such that \eqref{phi intro eq},
  i.\,e.
  \[\phi'(r) =\left(1+\phi^2(r)\right) \cdot
    \left(1-\frac{n-1}r\phi(r)\right),\] has a solution
  $\phi\in C^1([0,\epsilon])$ with $\phi(0)=0$ such
  that the limits
  \[\lim\limits_{r\downarrow0} \frac{\phi(r)}r \quad \text{and}\quad
    \lim\limits_{r\downarrow0} \phi'(r)\] exist and coincide.
\end{theorem}
\begin{proof}
  In view of the explicit grim reaper solution, we may assume
  $n\ge2$.\par
  We wish to apply Theorem \ref{1r sing ex thm} and rewrite
  \eqref{phi intro eq} as
  \begin{align*}
    \dot\phi(r)
    &=f\left(r,\frac{\phi(r)}r\right),\quad r\in(0,\epsilon),
      \intertext{where}
      f(s,a)
    &:=\left(1+s^2a^2\right)(1-(n-1)a).              
  \end{align*}
  Let $R>0$ be as in Corollary \ref{approx poly fct cor}. 
  Consider $S>0$ to be fixed later.  Let $a,b,s\in\R$ with
  $|a|,|b|\le 2R$ and $0\le s\le S$. We have to verify the uniform
  Lipschitz condition for $f$ and estimate
  \begin{align*}
    |f(s,a)-f(s,b)|
    &=\left|\left(1+s^2a^2\right)(1-(n-1)a)
      -\left(1+s^2b^2\right)(1-(n-1)b)\right|\umbruch\\
    &\le(n-1)|a-b|+s^2\left|a^2-b^2\right|
      +(n-1)s^2\left|a^3-b^3\right|\umbruch\\
    &=(n-1)|a-b| +s^2|a-b|\cdot|a+b|\\
    &\equad\quad {}+(n-1)s^2|a-b|\cdot
      \left|a^2+ab+b^2\right|\umbruch\\  
    &\le\left((n-1)+2S^2(2R)+3(n-1)S^2(2R)^2\right)\cdot |a-b|.
  \end{align*}
  For $S>0$ fixed sufficiently small and $L=n$, we obtain
  \begin{align*}
    |f(s,a)-f(s,b)|
    &\le L\cdot|a-b|. 
  \end{align*}
  We choose $p=n+1>\max\{L,1\}=n$. Now, Corollary \ref{approx poly fct
    cor} yields an approximative solution $h$ on $[0,\sigma]$ for some
  $\sigma\in(0,S]$ and Theorem \ref{1r sing ex thm} implies the
  existence of a solution $\phi\in C^1([0,\sigma])$ of \eqref{phi
    intro eq} with $\phi(0)=0$. Finally, Lemma \ref{ur cont lem 2} and
  the continuity of $\phi'$ imply that
  \[\lim\limits_{r\downarrow0}\frac{\phi(r)}r
    =\phi'(0) =\lim\limits_{r\downarrow0}\phi'(r)\] as claimed. 
\end{proof}

\section{Existence based on the shooting technique}
\label{shooting technique sec}
We wish to solve \eqref{phi intro eq} up to $r=0$, where it becomes
singular. Solutions shall fulfil the regularity conditions of
Proposition \ref{reg orig prop} so that in particular the term
$\frac{\phi(r)}r$ does not become singular. In order to find such a
solution, we start from an interval of possible initial values of
$\phi$ at some positive $r$. Then we try to solve \eqref{phi intro eq}
up to $r=0$ and discard initial values for which $\phi(r)$ becomes
unbounded. In this way, we ``shoot'' with solutions towards $r=0$ and
try to find the correct initial value, see Figure \ref{shooting.forward.pic}.
\begin{figure}[h]
  \includegraphics[width=0.8\linewidth]{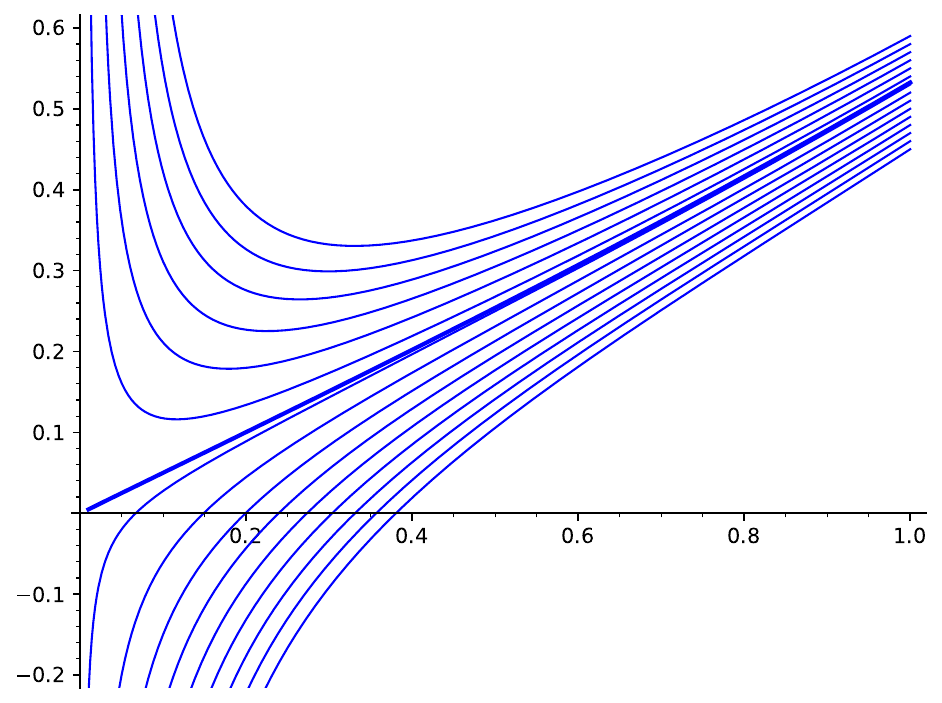}
  \caption{Shooting method. We see a family of solutions with
    different inital values at $r=1$ and a solution that corresponds to
    the translator as a thick line.}
  \label{shooting.forward.pic}
\end{figure}

In order to avoid working with the singular equation on a bounded
interval, we equivalently rewrite \eqref{phi intro eq} as a regular
equation on an unbounded interval.
\begin{lemma}
  Let $R>0$.  The differential equation \eqref{phi intro eq} for
  $r\in(0,R)$ is equivalent to the differential equation
  \begin{equation}
    \label{translat ode psi eq}
    \psi'(r)=\left(1+\psi^2(r)\right)\cdot
    \left((n-1)\cdot\psi(r)-e^{-r}\right)  
  \end{equation}
  for $\psi(r):=\phi(e^{-r})$ and $r\in(-\log R,\infty)$. 
\end{lemma}
\begin{proof}
  We differentiate the definition of $\psi$ and employ \eqref{phi
    intro eq} to obtain
  \begin{align*}
    \psi'(r)
    =&\,-e^{-r}\phi'\left(e^{-r}\right)\umbruch\\
    =&\,-e^{-r}\left(1+\phi^2\left(e^{-r}\right)\right)\cdot
       \left(1-\frac{n-1}{e^{-r}}
       \phi\left(e^{-r}\right)\right)\umbruch\\ 
    =&\,\left(1+\psi^2(r)\right)\cdot
       \left((n-1)\cdot\psi(r)-e^{-r}\right)
  \end{align*}
  as claimed. 
\end{proof}

In terms of $\psi$, the unique existence up to $r=0$ can be
reformulated as follows:
\begin{theorem}
  \label{tr mcf ex ori psi thm}
  Let $n\ge2$.  Then there exists a unique $a\in\R$ such that a
  solution $\psi$ to \eqref{translat ode psi eq} with $\psi(0)=a$
  exists on the interval $[0,\infty)$. Moreover, $\psi$ fulfils the
  regularity condition of Corollary \ref{reg orig exp cor}.
\end{theorem}
In view of Lemma \ref{existence large r lem}, it is no restriction to
choose $-\log R=0$. The following arguments can easily adapted to
other values of $R$, in particular in the case $R\in(0,1)$ and
$e^{-r}\le1$. \par
We will split the proof of Theorem \ref{tr mcf ex ori psi thm} into
several lemmata and prove those in this section.
\par In the following Lemma \ref{tr ori mcf eps lem} we will assume
that $\phi(r_0)\in[0,1]$. In order to establish the existence on
$[0,\infty)$, we do not need to consider other initial values: Lemma
\ref{tr mcf ori up bd lem} ensures that
$\phi(r_0)\le\frac{e^{-{r_0}}}{n-1}\le 1$ and Lemma \ref{tr mcf ori lo
  bd lem} yields $\psi(r_0)>0$.

\begin{lemma}
  \label{tr ori mcf eps lem}
  Let $n\ge2$. Then there exists $\epsilon=\epsilon(n)>0$ such that
  for any $r_0\ge0$ any maximal solution $\psi$ to \eqref{translat ode
    psi eq} with $\psi(r_0)\in[0,1]$ exists at least on the interval
  $[r_0,r_0+\epsilon]$.
\end{lemma}
\begin{proof}
  According to the maximal existence theorem for ordinary differential
  equations, the solution $\psi$ exists on some maximal interval
  $[r_0,r_1)$, $r_0<r_1\le\infty$, and, if $r_1<\infty$, $\psi(r)$
  becomes unbounded as $r\uparrow r_1$.  For $r\ge0$, we have the
  estimate
  \[\left(1+\psi^2(r)\right)\cdot (n-1)\cdot\psi(r) \ge\psi'(r)\ge
    \left(1+\psi^2(r)\right)\cdot\left((n-1) \cdot\psi(r)-1\right).\]
  In order to derive bounds for $\psi$, we consider solutions
  $\psi_\pm$ to
  \[
    \begin{cases}
      \psi_+'(r)=\left(1+\psi_+^2(r)\right)\cdot (n-1)\cdot\psi_+(r),\\
      \psi_+(r_0)=1
    \end{cases}
  \]
  and
  \[
    \begin{cases}
      \psi_-'(r)=\left(1+\psi_-^2(r)\right)\cdot
      \left((n-1)\cdot\psi_-(r)-1\right),\\ 
      \psi_-(r_0)=0,
    \end{cases}
  \]
  respectively. According to the existence theorem for ordinary
  differential equations, there exists $\epsilon=\epsilon(n)>0$ such
  that $\psi_+$ and $\psi_-$ exist at least on the interval
  $[r_0,r_0+2\epsilon]$. \par Usually, we find such an $\epsilon$ that
  depends on $n$ and $r_0$. In our situation, however, $\epsilon$ only
  depends on $n$, because for any solution $\psi_\pm$, the function
  $r\mapsto\psi_\pm(r-R)$, $R\in\R$ is another solution to the
  respective equation. \par
  If $r_1\ge r_0+\epsilon$, the lemma follows. Therefore, we may
  assume that $r_1<r_0+\epsilon$.  Now, the comparison theorem for
  ordinary real-valued differential equations of first order implies
  that $\psi_+(r)\ge \psi(r)\ge \psi_-(r)$ for all $r\in[r_0,r_1)$. In
  particular, $\psi(r)$ does not become unbounded as $r\to r_1$. This
  contradicts the fact that $\psi$ exists only up to $r_1$ and hence
  has to become unbounded as $r\uparrow r_1$.
\end{proof}

We will derive bounds for solutions to \eqref{translat ode psi eq}
that exist on $[0,\infty)$.
\begin{lemma}
  \label{tr mcf ori up bd lem}
  Let $n\ge2$ and let $\psi$ be a solution to \eqref{translat ode psi
    eq}.
  \begin{enumerate}[label=(\roman*)]
  \item If $(n-1)\cdot\psi(r_0)-e^{-r_0}\ge0$ for some $r_0\ge0$,
    then we obtain \[(n-1)\cdot\psi(r)-e^{-r}>0\] for any $r\ge r_0$ for
    which $\psi$ exists. 
  \item If $\psi$ exists on $[0,\infty)$,
    then \[(n-1)\cdot\psi(r)-e^{-r}\le0\] for all $r\ge0$.
  \end{enumerate}
\end{lemma}
\begin{proof}
  We define $w(r):=(n-1)\cdot\psi(r)-e^{-r}$ and observe that
  \[\psi'(r) =\left(1+\psi^2(r)\right)\cdot w(r).\] This yields the
  differential inequality
  \begin{equation}
    \begin{split}
      \label{tr mcf ori up bd w est}
      w'(r)&=(n-1)\cdot\psi'(r)+e^{-r}\\
      &>(n-1)\cdot\left(1+\psi^2(r)\right)\cdot w(r).
    \end{split}
  \end{equation}
  for $w$. Now, we can address our assertions. 
  \begin{enumerate}[label=(\roman*)]
  \item This is a direct consequence of the Estimate \eqref{tr mcf ori
      up bd w est}.
  \item We argue by contradiction and assume that
    \[\epsilon:=w(r_1)=(n-1)\cdot\psi(r_1)-e^{-r_1}>0\] for some
    $r_1\ge0$. It follows from \eqref{tr mcf ori up bd w est} that
    $w(r)\ge\epsilon$ for all $r\ge r_1$. Therefore, we deduce that
    \begin{align*}
      w'(r)\ge&\,(n-1)\cdot\psi^2(r)\cdot w(r)\umbruch\\
      =&\,(n-1)\cdot\left(\frac{w(r)+e^{-r}}{n-1}\right)^2\cdot
         w(r)\umbruch\\
      \ge&\,\frac1{n-1}\cdot w^3(r). 
    \end{align*}
    The solution $W$ to $W'(r)=\frac1{n-1}\cdot W^3(r)$ with
    $W(r_1)=\epsilon$ is given by
    \[W(r)=\frac1{\sqrt{\frac1{\epsilon^2}-\frac 2{n-1}(r-r_1)}}.\] In
    particular, it exists only up to
    $r_1+\frac{n-1}{2\epsilon^2}=:r_2$ and tends to infinity as
    $r\uparrow r_2$. Hence, due to the comparison theorem for ordinary
    differential equations, we obtain $w\ge W$ and deduce that $w$ can
    exist at most up to $r_2<\infty$. We arrive at a
    contradiction. This yields our second assertion.  \qedhere
  \end{enumerate}
\end{proof}

\begin{lemma}
  \label{tr mcf ori lo bd lem}
  Let $n\ge2$ and let $\psi$ be a solution to \eqref{translat ode psi
    eq}. 
  \begin{enumerate}[label=(\roman*)]
  \item If $\psi(r_0)\le0$ for some $r_0\ge0$, then $\psi(r)<0$ for
    all $r\ge r_0$ for which $\psi$ exists. 
  \item If $\psi$ exists on $[0,\infty)$, then \[\psi(r)>0\] for all
    $r\ge0$.
  \end{enumerate}
\end{lemma}
\begin{proof}
  \neueZeile
  \begin{enumerate}[label=(\roman*)]
  \item From \eqref{translat ode psi eq}, we obtain directly that
    $\psi'(r)<0$ for all $r$ such that $\psi(r)\le0$. This yields the
    first part of the lemma. 
  \item Assume that there exists some $r_1\ge0$ such that
    $\psi(r_1)\le0$. Then, as above, according to \eqref{translat ode
      psi eq}, we obtain $\psi'(r_1)<0$. Hence, we may, by increasing
    $r_1$ if necessary, assume without loss of generality that
    $\psi(r_1)\le-\epsilon$ for some $\epsilon>0$. Using
    \eqref{translat ode psi eq} once again, we deduce that
    $\psi(r)\le-\epsilon$ for all $r\ge r_1$. Hence, dropping some
    negative terms in the equation, we deduce
    that \[\psi'(r)\le(n-1)\cdot\psi^3(r) \le\psi^3(r)\] for
    $r\ge r_1$. The solution $\Psi$ to $\Psi'(r)=\Psi^3(r)$ with
    $\Psi(r_1)=-\epsilon$ is given
    by \[\Psi(r)=-\frac1{\sqrt{\frac1{\epsilon^2}-2(r-r_1)}}.\]
    Similarly as in the proof of Theorem \ref{tr mcf ori up bd lem},
    $\Psi$ exists only up to $r_1+\frac1{2\epsilon^2}=:r_2$ and hence,
    due to the comparison principle, $\psi$ exists at most up to
    $r_2<\infty$. This contradiction finishes the proof of the lemma.
    \qedhere
  \end{enumerate}
\end{proof}

\begin{proof}[Proof of Theorem \ref{tr mcf ex ori psi thm}: Existence]
  We define $I_k\subset[0,1]$ as the set of all points $a\in[0,1]$
  such that the solution $\psi$ to \eqref{translat ode psi eq} with
  $\psi(0)=a$ exists at least up to
  $k\cdot\epsilon\equiv k\cdot\epsilon(n)$ with $\epsilon(n)$ as in
  Lemma \ref{tr ori mcf eps lem} and
  \[0\le\psi(k\cdot\epsilon)\le\frac1{n-1}e^{-k\cdot\epsilon}\le1.\]
  According to Lemma \ref{tr mcf ori up bd lem} and Lemma \ref{tr mcf
    ori lo bd lem}, we obtain $I_{k+1}\subset I_k$ for any
  $k\in\N$. Moreover, the comparison theorem, implies that each set
  $I_k$, $k\in\N$, is an interval. Furthermore, due to the continuous
  dependence of solutions on initial values for ordinary differential
  equations, we see that each interval $I_k$, $k\in\N$, is closed and
  hence compact. \par
  \emph{Surjectivity:} We now claim that for any
  $b\in\left[0,\frac1{n-1}e^{-k\cdot\epsilon}\right]$, there is some
  $a\in I_k\subset[0,1]$ such that the solution $\psi$ to
  \eqref{translat ode psi eq} with $\psi(0)=a$ fulfils
  $\psi(k\cdot\epsilon)=b$. We use induction on $k$ to prove this
  claim. For $k=0$, we may take $a=b$. Assume now, that the claim is
  already established for some $k_0\in\N$. Let $\psi_{k_0,+}$ be the
  solution to \eqref{translat ode psi eq} such that
  $\psi_{k_0,+}(k_0\cdot\epsilon) =\frac1{n-1}e^{-k_0\cdot\epsilon}$
  and let $\psi_{k_0,-}$ be the solution to \eqref{translat ode psi
    eq} such that $\psi_{k_0,-}(k_0\cdot\epsilon) =0$. According to
  Lemma \ref{tr ori mcf eps lem}, $\psi_{k_0,+}$ and $\psi_{k_0,-}$
  exist at least up to $r=(k_0+1)\cdot\epsilon$. Moreover, Lemma
  \ref{tr mcf ori up bd lem} implies that
  \[\psi_{k_0,+}((k_0+1)\cdot\epsilon)\ge
    \frac1{n-1}e^{-(k_0+1)\cdot\epsilon}\] and Lemma \ref{tr mcf ori
    lo bd lem} yields that
  $\psi_{k_0,-}((k_0+1)\cdot\epsilon)\le0$. We set
  $a_{k_0,+}:=\psi_{k_0,+}(0)$ and $a_{k_0,-}:=\psi_{k_0,-}(0)$. Then
  we infer from the comparison theorem and from the maximal existence
  theorem, that solutions $\psi$ to \eqref{translat ode psi eq} with
  $\psi(0)\in[a_{k_0,-},a_{k_0,+}]$ exist at least up to
  $r=(k_0+1)\cdot\epsilon$. We denote the solution $\psi$ with
  $\psi(0)=a$ by $\psi_a$. Our definitions are such that
  $\psi_{a_{k_0,+}} =\psi_{k_0,+}$ and
  $\psi_{a_{k_0,-}} =\psi_{k_0,-}$. Hence, we obtain that
  \[\psi_{a_{k_0,-}}((k_0+1)\cdot\epsilon)\le0
    \le\frac1{n-1}e^{-(k_0+1)\cdot\epsilon}\le
    \psi_{a_{k_0,+}}((k_0+1)\cdot\epsilon).\] Due to continuous
  dependence on the initial value, the values
  $\psi_a((k_0+1)\cdot\epsilon)$ depend continuously on
  $a$. Therefore, we find for each
  $b\in\left[0,\frac1{n-1}e^{-(k_0+1)\cdot\epsilon}\right]$ some
  $a\in[a_{k_0,-},a_{k_0,+}]$ such that
  $\psi_a((k_0+1)\cdot\epsilon) =b$. It is clear from the definition
  of $I_{k_0+1}$ that $a\in I_{k_0+1}$ for $a$ as above. This is
  illustrated in Figures \ref{shooting.fig} and \ref{shooting.fig.log}. 
  \begin{figure}[ht]
    \includegraphics[width=0.8\linewidth]{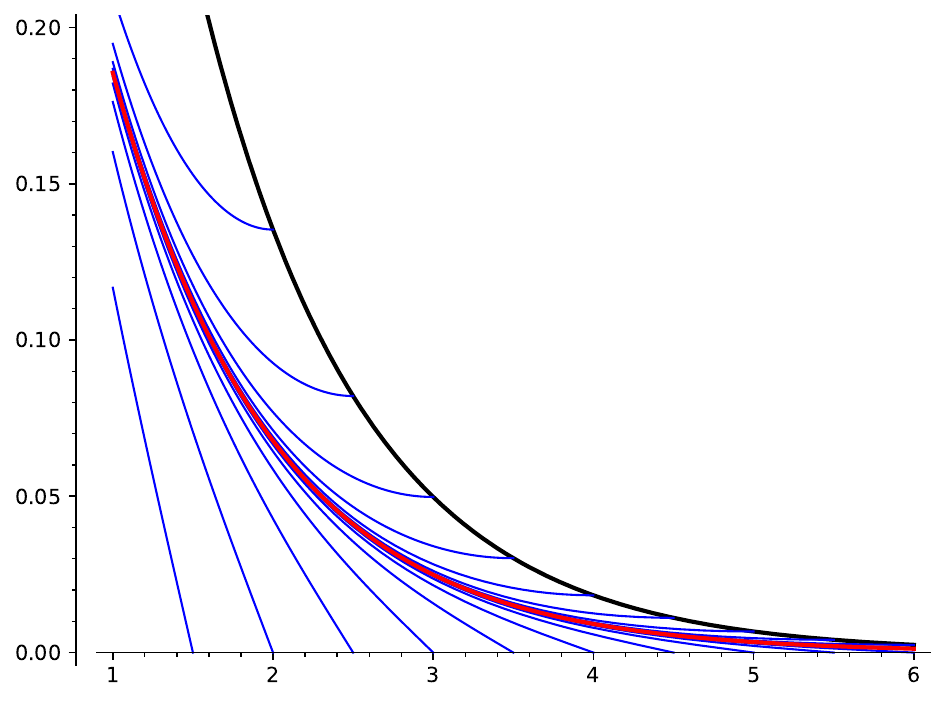}
    \caption{Approximative solutions to \eqref{translat ode psi
        eq}. In this figure, the black curve is the graph of
      $r\mapsto\frac{e^{-r}}{n-1}$, see Lemma \ref{tr mcf ori up bd
        lem}. All other curves in the picture solve \eqref{translat
        ode psi eq}. The red curve corresponds to a translator which
      is smooth at the origin. The blue curves are $\psi_{k_0,-}$ and
      $\psi_{k_0,+}$ for $\epsilon=\frac12$ and
      $k_0=3,4,\ldots,12$. They pass through $(r,0)$ and
      $\left(r,\frac{e^{-r}}{n-1}\right)$, respectively, for
      $r=\epsilon\cdot k_0=1.5,2.0,\ldots,6.0$. In order to draw them,
      we use these points as initial values and solve \eqref{translat
        ode psi eq} backwards up to $r=1.0$.}
    \label{shooting.fig}
  \end{figure}
  
  \begin{figure}[ht]
    \includegraphics[width=0.8\linewidth]{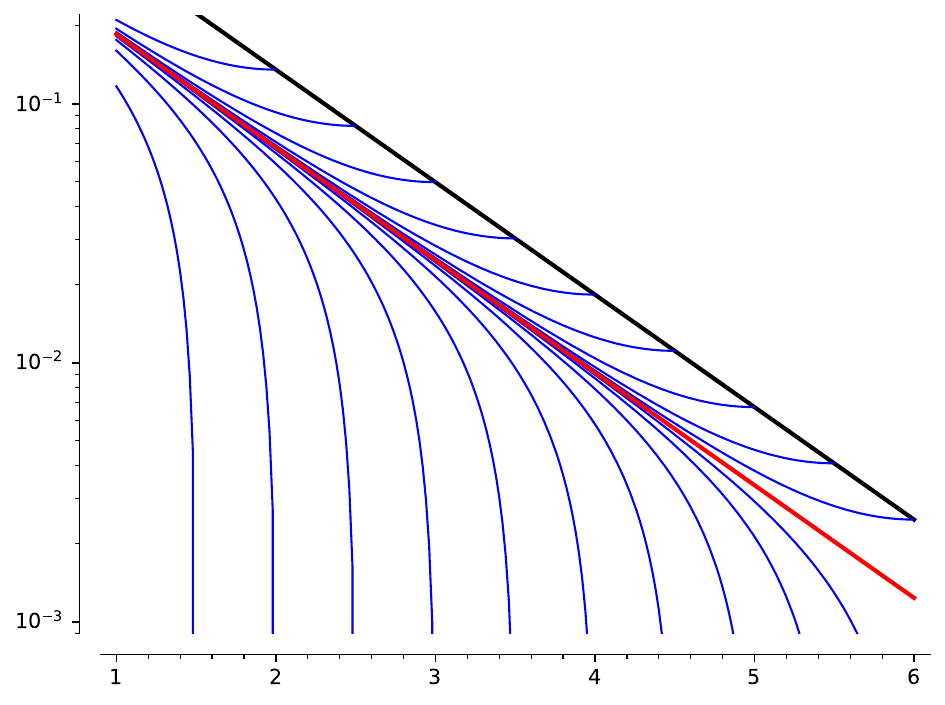}
    \caption{Approximative solutions to \eqref{translat ode psi eq} as
      in Figure \ref{shooting.fig}, in logarithmic scaling.}
    \label{shooting.fig.log}
  \end{figure}
  \begin{align*}
    \displaybreak[1]
  \end{align*}
  \par This last claim implies in particular that each $I_k$,
  $k\in\N$, is non-empty. As a sequence of decreasing non-empty
  compact intervals has a non-empty intersection, we find some
  $a\in[0,1]$ such that $a\in I_k$ for each $k\in\N$. This means that
  $\psi_a$ exists on the interval $[0,\infty)$. Hence, we have
  established the existence part of Theorem \ref{tr mcf ex ori psi
    thm}.
\end{proof}

\begin{proof}[Proof of Theorem \ref{tr mcf ex ori psi thm}:
  Uniqueness] Let $\psi,\Psi\colon[0,\infty)\to\R$ be solutions to
  \eqref{translat ode psi eq}. The Equation \eqref{translat ode psi
    eq} is a first order real valued ordinary differential
  equation. Therefore, solutions have to coincide or $\Psi-\psi$ is
  everywhere positive or negative. Hence, we may assume without loss
  of generality for the uniqueness proof that $\Psi(r)>\psi(r)$ for
  all $r\ge0$.  According to Lemma \ref{tr mcf ori up bd lem} and
  Lemma \ref{tr mcf ori lo bd lem}, we have the
  estimates
  \begin{equation}
    \label{psi Psi exp eq}
    0\le\psi(r)\le\Psi(r)\le \frac{e^{-r}}{n-1}
  \end{equation}
  for all $r\ge0$. We consider the difference $w:=\Psi-\psi$. Using
  \eqref{translat ode psi eq} and the estimate above, we obtain
  \begin{align*}
    w'(r)=&\,\Psi'(r)-\psi'(r)\umbruch\\
    =&\,\left(1+\Psi^2(r)\right)\cdot
       \left((n-1)\cdot\Psi(r)-e^{-r}\right)\\
          &\,-\left(1+\psi^2(r)\right)\cdot
            \left((n-1)\cdot\psi(r)-e^{-r}\right)\umbruch\\
    =&\,\int\limits_0^1 \frac d{d\tau} \left\{\left(1+(\tau\Psi(r)
       +(1-\tau)\psi(r))^2\right)\right.\\
          &\,\quad\left.\cdot
            \left((n-1)\cdot(\tau\Psi(r)+(1-\tau)\psi(r))
            -e^{-r}\right)\right\}\,d\tau\umbruch\\
    =&\,\int\limits_0^1 2(\tau\Psi(r)+(1-\tau)\psi(r))\\
    &\,\quad\cdot
       \left((n-1)\cdot(\tau\Psi(r)+(1-\tau)\psi(r))
       -e^{-r}\right)\,d\tau\cdot w(r)\\
    &\,+\int\limits_0^1
      \left(1+(\tau\Psi(r)+(1-\tau)\psi(r))^2\right)(n-1)\,d\tau\cdot
      w(r)\umbruch\\
    \ge&\,(n-1)\cdot w(r) -c\cdot e^{-r}\cdot w(r)\umbruch\\
    \ge&\,\left(n-\frac32\right)\cdot w(r) \ge\frac12 w(r), 
  \end{align*}
  where we have assumed for the last estimate that $r\ge r_0$ for some
  sufficiently large $r_0>0$. We deduce that
  $w(r)\ge e^{\frac12(r-r_0)}w(r_0)$. This contradicts the estimates
  above in \eqref{psi Psi exp eq} unless $w(r_0)=0$. Hence, we deduce
  that $\Psi(r_0)=\psi(r_0)$ and conclude that $\Psi=\psi$ as claimed.
\end{proof}

\begin{remark}

  In the proof of the uniqueness statement in Theorem \ref{tr mcf ex
    ori psi thm}, we have used the fundamental theorem of calculus to
  derive a differential inequality for the difference
  $w:=\Psi-\psi$. This technique can be applied to many differential
  equations, in particular to differential equations of the
  form \[\psi'(r)=f(r,\psi(r))\] with $f\in C^1$. \par In our special
  case, however, elementary algebra yields the same result. Under
  assumptions as in that proof, the relevant calculation becomes
  \begin{align*}
    w'(r)
    &=\Psi'(r)-\psi'(r)\umbruch\\
    &=\left(1+\Psi^2(r)\right)\left((n-1)\Psi(r)-e^{-r}\right)\\
    &\equad -\left(1+\psi^2(r)\right)
      \left((n-1)\psi(r)-e^{-r}\right)\umbruch\\ 
    &=(n-1)(\Psi(r)-\psi(r))\\
    &\equad+(n-1)\left(\Psi^3(r)-\psi^3(r)\right)\\
    &\equad-e^{-r}\left(\Psi^2(r)-\psi^2(r)\right)\umbruch\\
    &=(n-1)w(r)\\
    &\equad+(n-1)\left(\Psi^2(r)+\Psi(r)\psi(r)+\psi^2(r)\right)
      (\Psi(r)-\psi(r))\\
    &\equad-e^{-r}(\Psi(r)+\psi(r)) (\Psi(r)-\psi(r))\umbruch\\
    &=(n-1)w(r)\\
    &\equad+(n-1)\left(\Psi^2(r)+\Psi(r)\psi(r)+\psi^2(r)\right)
      w(r)\\
    &\equad-e^{-r}(\Psi(r)+\psi(r))w(r)\umbruch\\
    &\ge(n-1)w(r) +0-2e^{-r}\frac{e^{-r}}{n-1}w(r)\umbruch\\
    &\ge\left(n-\frac32\right)w(r).  
  \end{align*}
  The resulting inequality is the same as in the uniqueness proof
  above.
\end{remark}

\begin{proof}[Proof of Theorem \ref{tr mcf ex ori psi thm}: Regularity
  condition] According to Corollary \ref{reg orig exp cor}, we have to
  prove that
  \[\lim\limits_{r\to\infty} e^r\psi(r)\quad\text{and}\quad
    \lim\limits_{r\to\infty} e^r\psi'(r)\] exist and sum up to
  zero.
  \begin{enumerate}[label=(\roman*)]
  \item For the proof that the first limit exists, we define
    $w(r):=e^r\psi(r)$. From Lemma \ref{tr mcf ori up bd lem} and
    Lemma \ref{tr mcf ori lo bd lem} we infer that
    \[0\le \psi(r)\le \frac{e^{-r}}{n-1} \quad\text{and,
        equivalently,}\quad 0\le w(r)\le\frac1{n-1}\] for all
    $r\ge0$. Next, we derive a differential equation for $w$. We get
    \begin{align*}
      w'(r)=&\,e^r\psi(r)+e^r\psi'(r)\umbruch\\
      =&\,e^r\psi(r)+e^r\left(1+\psi^2(r)\right)
         \left((n-1)\psi(r)-e^{-r}\right)\umbruch\\
      =&\,e^r\psi(r) +(n-1)e^r\psi(r) -1 +(n-1)e^r\psi^3(r)
         -\psi^2(r)\umbruch\\
      =&\,nw(r) +(n-1)\psi^2(r)w(r)-1-\psi^2(r).
    \end{align*}
    Heuristically, we expect that the terms involving $\psi$ can be
    neglected for large values of $r$ and we obtain $u'(r)=nu(r)-1$,
    where we have used $u$ instead of $w$ to avoid
    confusion. Solutions are given by $u(r)=\frac1n+ae^{nr}$. Such
    solutions are only bounded if $a=0$. We conjecture that $w$
    behaves similarly as $u$. Therefore, we wish to show that
    \begin{equation}
      \label{w lim 1n eq}
      \lim\limits_{r\to\infty} w(r)=\frac1n.
    \end{equation}
    \begin{itemize}
    \item \textbf{Bound from above:} Let $\epsilon>0$. Fix $r_0\ge0$
      large enough to ensure that $e^{-2r_0}\le\epsilon$. Let us
      assume for the rest of this part of the proof that $r\ge
      r_0$. Then we obtain
      $\psi^2(r)\le \frac{e^{-2r}}{(n-1)^2}\le e^{-2r}\le\epsilon$.
      Now, consider $r$ such that $w(r)\ge\frac1n+\epsilon$. We deduce
      that
      \[w'(r)\ge n\left(\frac1n +\epsilon\right)+0-1-\epsilon
        =n\epsilon-\epsilon =(n-1)\epsilon.\] Therefore, $w$ continues
      to grow with $w'(r)\ge(n-1)\epsilon$ if it ever reaches or
      exceeds $\frac1n+\epsilon$. For large values of $r$, this
      contradicts $w(r)\le\frac1{n-1}$. Therefore, we have shown that
      $w(r)\le\frac1n+\epsilon$ for all $r\ge r_0$.
    \item \textbf{Bound from below:} We proceed similarly as in the
      proof of the bound from above. Fix $\epsilon\in\R$ such that
      $0<\epsilon<\frac1n$. Assume that $r_0$ is large enough such
      that $(n-1)\psi^2(r)\le\epsilon$ for all $r\ge r_0$.  Recall
      from Lemma \ref{tr mcf ori lo bd lem} that $0\le w(r)$. Now, we
      consider $r\ge r_0$ such that $w(r)\le\frac1n-\epsilon$ and wish
      to show that this is impossible for large values of $r$. We get
      \[w'(r)\le n\left(\frac1n-\epsilon\right)
        +\epsilon\left(\frac1n-\epsilon\right)-1+0
        \le-(n-1)\epsilon.\] Similarly as for the bound from above, we
      see that $w(r)\le\frac1n-\epsilon$ is preserved as $r$
      increases, hence $w'(r)\le-(n-1)\epsilon$ will eventually lead
      to $0>w(r)$.  This contradicts $w(r)\ge0$ and we infer that
      $w(r)\ge\frac1n-\epsilon$ for all $r\ge r_0$.
    \end{itemize}
    This finishes the proof of \eqref{w lim 1n eq}.
  \item In order to compute the limit of $e^r\psi'(r)$, we multiply
    \eqref{translat ode psi eq} with $e^r$ and use to results
    above. We obtain
    \begin{align*}
      e^r\psi'(r) =&\,\left(1+\psi^2(r)\right)\cdot\left((n-1)\cdot
                     e^r\psi(r)-1\right)\umbruch\\
      \overset{r\to\infty}\longrightarrow
                   &\,(1+0)\cdot\left((n-1)\frac1n-1\right)
                     =-\frac1n 
    \end{align*}
    as claimed. \qedhere
  \end{enumerate}  
\end{proof}

\section{Existence based on approximating problems}
\label{approx pbls sec}
In order to solve \eqref{phi intro eq}, we will consider it on
intervals of the form $\left[\frac1k,\infty\right)$, $k\in\N_{\ge1}$,
and pass to a limit as $k\to\infty$ to obtain a solution to the
singular equation \eqref{phi intro eq}. Here, our choice of the
initial value at $r=\frac1k$ is motivated by the expansion
$\phi(r)\approx\frac1n r$ for $r\approx 0$ derived in Remark \ref{expl
  asymp rem}.

Let $\phi_k\colon\left[\frac1k,\infty\right)\to\R$, $k\in\N_{\ge1}$, be
solutions to the initial value problems
\begin{equation}
  \label{1k approx pbl}
  \begin{cases}
    \phi'(r)
    =\left(1+\phi^2(r)\right)\left(1-\frac{n-1}r\phi(r)\right),
    &r\ge\frac1k,\\
    \phi\left(\frac1k\right)=\frac1n\frac1k.&
  \end{cases}
\end{equation}

In the following results, we will always assume that $\phi_k$ is a
maximal solution to \eqref{1k approx pbl} and only consider
$r\ge\frac1k$.
\begin{lemma}
  \label{phi first bd lem}
  Let $n\ge2$ and let $\phi_k$ be a maximal solution to \eqref{1k approx
    pbl}. Then we have
  \[0<\phi_k(r)<\frac1{n-1}r\] for all $r\ge\frac1k$, where $\phi_k$ is
  defined. 
\end{lemma}
\begin{proof}
  We will write $\phi$ instead of $\phi_k$ in the proof.  For
  $r=\frac1k$, we have $0<\phi(r)<\frac1{n-1}r$. For $\phi(r)=0$, the
  right-hand side of the equation equals one, hence $\phi(r)>0$ is
  preserved. Similarly, for $\phi(r)=\frac1{n-1}r$, the right-hand
  side of the equation vanishes, whereas the derivative of
  $\frac1{n-1}r$ is positive. Hence $\phi(r)<\frac1{n-1}r$ is also
  preserved.
\end{proof}

As a consequence from standard theory for ordinary differential
equation concerning a characterisation of the maximal existence
interval, we get
\begin{corollary}
  Let $n\ge2$.  The maximal solution $\phi_k$ to \eqref{1k approx pbl} is
  indeed defined for all $r\ge\frac1k$.
\end{corollary}

Consider any interval $(a,b)$ compactly contained in
$(0,\infty)$. Then all $\phi_k$, if defined, are uniformly
bounded. Due to the equation, this applies also to
$\phi_k'$. Therefore, Arzel\`a-Ascoli yields a subsequence that
converges uniformly to a function $\phi$. As \eqref{1k approx pbl} can
equivalently be rewritten as an integral equation, $\phi$ also solves
\eqref{1k approx pbl}. We apply this with intervals of the form
$\left(\frac1l,l\right)$, $l\in\N$, consider a diagonal sequence and
obtain a solution $\phi\colon(0,\infty)\to\R$ of \eqref{phi intro eq}.

\begin{lemma}
  Let $n\ge2$ and let $\phi_k$ be a solution to \eqref{1k approx
    pbl}. Let
  $\epsilon>0$. Then \[\frac1nr<\phi_k(r)<\frac1nr+\epsilon r^2\]
  applies to all $r$ such that $\frac1k\le r\le n(n-1)^3\epsilon$.
\end{lemma}
\begin{proof}
  Once again, we write $\phi$ instead of $\phi_k$ in the proof. 
  Let $w(r)=\phi(r)-\frac1nr$. We have
  $w\left(\frac1k\right)=0$. Using \eqref{1k approx pbl}, we get
  \begin{align}
    w'(r)
    &=\phi'(r)-\frac1n\nonumber\umbruch\\
    &=\left(1+\phi^2(r)\right)\left(1-\frac{n-1}r\phi(r)\right)
      -\frac1n\nonumber\umbruch\\
    & \label{phi mrn eq}
      =\left(1+\phi^2(r)\right)\left(1-\frac{n-1}rw(r) -\frac{n-1}r\frac1n
      r\right) -\frac1n\umbruch\\
    &=\left(1+\phi^2(r)\right)\left(-\frac{n-1}rw(r) +\frac1n
      \right) -\frac1n\nonumber\umbruch\\
    &=-\left(1+\phi^2(r)\right)\frac{n-1}rw(r)
      +\frac1n\phi^2(r).\nonumber
  \end{align}
  As in Lemma \ref{phi first bd lem}, we see immediately that $w(r)>0$
  for all $r>\frac1k$, corresponding to $\phi(r)>\frac1nr$. \par If
  $\phi(r)=\frac1n r+\epsilon r^2$, we use our bounds from Lemma
  \ref{phi first bd lem} and \eqref{phi mrn eq} to deduce that
  \[w'(r)\le-(1+0)(n-1)\epsilon r +\frac1n\frac1{(n-1)^2}r^2.\]
  Therefore, we see that $w(r)\le\epsilon r^2$ is preserved as long as
  $r\le n(n-1)^3\epsilon$.
\end{proof}

This estimate ensures that $\frac1nr\le\phi(r)\le\frac1nr+\epsilon
r^2$ for small values of $r$, depending on $\epsilon$. Thus $\phi$ can
be extended continuously to $r=0$ by $\phi(r)=0$ and we see
that \[\lim\limits_{r\downarrow0}\frac{\phi(r)}r=\frac1n.\]

As we wish to find a solution $\phi$ that corresponds to a
$C^2$-hypersurface at $r=0$, we have to verify the condition in
Proposition \ref{reg orig prop}. This amounts to showing that
$\lim\limits_{r\downarrow0}\phi'(r)=\frac1n$. In order to prove this,
we use \eqref{phi intro eq} for $r>0$ and limit theorems to deduce
that
\[\lim\limits_{r\downarrow0}\phi'(r)
  =(1+0)\left(1-(n-1)\lim\limits_{r\downarrow0} \frac{\phi(r)}r\right)
  =1-\frac{n-1}n =\frac1n.\]  Therefore,
  $\phi$ yields a rotationally symmetric translating solution to mean
  curvature flow of class $C^2$.

\section{Existence based on a regularised equation}
\label{reg eq sec}
Let $n\ge2$. Consider for $\epsilon>0$ the initial value problems
\begin{equation}
  \label{eps ode}
  \begin{cases}
    \phi'(r) =\left(1+\phi^2(r)\right)\left(1-\frac{n-1}{r+\epsilon}
      \phi(r)\right), & r\ge0,\\
    \phi(0)=0.
  \end{cases}
\end{equation}

Solutions exist for all $r\ge0$. 
\begin{lemma}
  Let $n\ge2$. Then for each $\epsilon>0$, the maximal solution
  $\phi_\epsilon$ to \eqref{eps ode} exists on the entire interval
  $[0,\infty)$. 
\end{lemma}
\begin{proof}
  The initial condition ensures that $\phi_\epsilon$ exists on some
  maximal interval $[0,r_\epsilon)$,
  $r_\epsilon\in(0,\infty]$. Clearly, there exists exactly one
  solution for fixed $\epsilon$. We see that $\phi'_\epsilon(r)>0$
  whenever $\phi_\epsilon(r)=0$ and deduce that $\phi_\epsilon\ge0$ as
  long as it exists. Moreover
  $\phi_\epsilon(r)\ge\frac{r+\epsilon}{n-1}$ implies that
  $\phi'_\epsilon(r)\le0$. We obtain
  $0\le\phi_\epsilon(r)\le\frac{r+\epsilon}{n-1}$. According to the
  characterisation of the maximal existence interval $[0,r_\epsilon)$,
  we have $r_\epsilon=\infty$ or $|\phi_\epsilon(r)|\to\infty$ as
  $r\uparrow r_\epsilon$. Thus our bounds on $\phi_\epsilon(r)$ ensure
  that $r_\epsilon=\infty$ for all $\epsilon>0$ as claimed.
\end{proof}

For fixed $r>0$, the solutions $\phi_\epsilon(r)$ are monotone in
$\epsilon$. 
\begin{lemma}
  Let $n\ge2$, fix $0<\epsilon_1<\epsilon_2$ and let
  $\phi_{\epsilon_1}$ and $\phi_{\epsilon_2}$ be solutions to
  \eqref{eps ode} with $\epsilon=\epsilon_1$ and
  $\epsilon=\epsilon_2$, respectively. Then we
  have \[\phi_{\epsilon_1}(r)<\phi_{\epsilon_2}(r)\] for all $r>0$.
\end{lemma}
\begin{proof}
  We differentiate \eqref{eps ode} and obtain
  \begin{align*}
    \phi''(r)
    &=\left(1+2\phi(r)\phi'(r)\right)
      \left(1-\frac{n-1}{r+\epsilon}\phi(r)\right)\\  
    &\equad+\left(1+\phi^2(r)\right)
      \left(\frac{n-1}{(r+\epsilon)^2}\phi(r) 
      -\frac{n-1}{r+\epsilon}\phi'(r)\right). 
  \end{align*}
  Combining this with \eqref{eps ode} and the initial condition, we
  get for any $\epsilon>0$
  \[\phi_\epsilon(0)=0,\quad \phi'_\epsilon(0)=1\quad\text{and}\quad
    \phi''_\epsilon(0) =1-\frac{n-1}\epsilon\cdot1.\] Therefore, for
  fixed $0<\epsilon_1<\epsilon_2$, there exists some $\delta>0$ such
  that
  \[\phi_{\epsilon_1}(r) <\phi_{\epsilon_2}(r) \quad\text{for all
    }0<r<\delta.\] Assume now, that there exists some minimal $r_0>0$
  such that $\phi_{\epsilon_1}(r_0)=\phi_{\epsilon_2}(r_0)$. Using
  \eqref{eps ode} once again, we get
  $\phi'_{\epsilon_1}(r_0)<\phi'_{\epsilon_2}(r_0)$. However,
  $\phi_{\epsilon_1}(r)<\phi_{\epsilon_2}(r)$ for $0<r<r_0$ and
  equality for $r=r_0$ imply
  $\phi'_{\epsilon_1}(r_0)
  \ge\phi'_{\epsilon_2}(r_0)$. Contradiction. Our claim follows.
\end{proof}

The monotonicity and nonnegativity of
$\epsilon\mapsto\phi_\epsilon(r)$ for each $r\ge0$ allow to consider
the pointwise limit
\[\phi(r) :=\lim\limits_{\epsilon\downarrow0} \phi_\epsilon(r).\]
Assume from now on that $0<\epsilon\le1$. Then $\phi_\epsilon(r)$ is
uniformly bounded on any compact interval $I\subset (0,\infty)$. As
$r$ is also uniformly bounded below, \eqref{eps ode} implies uniform
bounds on $\phi'_\epsilon(r)$ on these intervals. We deduce locally
uniform convergence $\phi_\epsilon\rightrightarrows\phi$ for some
continuous function $\phi\colon(0,\infty)\to\R$. Employing the
equivalent integral equation to \eqref{eps ode}, we can pass to the
limit $\epsilon\downarrow0$ in this formulation. Thus
$\phi\in C^1((0,\infty))$ and it solves the differential equation
\eqref{phi intro eq}. It remains to study $\phi$ near the origin and
to establish $\phi(0)=0$ and in particular $C^2$-regularity of the
corresponding translator.

We start with a lower bound.
\begin{lemma}
  \label{phi eps lower bd lem}
  Let $n\ge2$ and let $\phi_\epsilon$ be a solution to \eqref{eps
    ode}. Then we have the estimate \[\phi_\epsilon(r)\ge\frac rn\]
  from below for all $r\ge0$. 
\end{lemma}
\begin{proof}
  We wish to show that $\psi(r):=\frac rn$ is a lower barrier,
  i.\,e.{} $\psi(0)\le \phi_\epsilon(0)$ and $\psi'(r)
  \le\left(1+\psi^2(r)\right)
  \left(1-\frac{n-1}{r+\epsilon}\psi(r)\right)$ for all $r\ge0$. A
  direct calculation yields
  \[1-\frac{n-1}{r+\epsilon}\psi(r)
    \ge1-\frac{n-1}{r+\epsilon}\frac{r+\epsilon}n
    =1-\frac{n-1}n=\frac1n.\] As $\psi'(r)=\frac1n$ and
  $1+\psi^2(r)\ge1$, our claim follows.
\end{proof}

The upper bound is more delicate. 
\begin{lemma}
  \label{phi eps upper bd lem}
  Let $n\ge2$. 
  Let $\phi_\epsilon$ be a solution to \eqref{eps ode}. Then we get
  the estimate 
  \[\phi_\epsilon(r) \le \frac{r+\epsilon}n
    +\left(\frac{r+\epsilon}n\right)^2\] from above for all
  $0\le r\le\frac n3$ and $0<\epsilon\le\frac n3$ and thus in
  particular for all $0<r,\epsilon\le\frac23$. 
\end{lemma}
\begin{proof}
  We proceed similarly as for the lower bound an wish to show
  that
  \[\psi_\epsilon(r) :=\frac{r+\epsilon}n
    +\left(\frac{r+\epsilon}n\right)^2\] is an upper barrier, i.\,e.{}
  $\psi_\epsilon(0)\ge\phi_\epsilon(0)$ and
  \[\psi'_\epsilon(r)\ge \left(1+\psi^2_\epsilon(r)\right)
    \left(1-\frac{n-1}{r+\epsilon}\psi_\epsilon(r)\right)\] for
  $0\le r\le\frac n3$, $0<\epsilon\le\frac n3$. While the first
  inequality is obviously true, we set $x:=\frac{r+\epsilon}n$ and
  rewrite the second inequality equivalently as
  \begin{align*}
    \psi'_\epsilon(r)
    &=\frac1n +2\frac1n\frac{r+\epsilon}n
      =\frac1n(1+2x)\umbruch\\
    &\overset!\ge \left(1+\psi_\epsilon^2(r)\right)
      \left(1-\frac{n-1}{r+\epsilon} \psi_\epsilon(r)\right)\umbruch\\   
    &=\left(1+x^2(1+x)^2\right)
      \left(1-(n-1)\frac1n(1+x)\right)\umbruch\\
    &=\left(1+x^2(1+x)^2\right) \frac1n(n-n-nx+1+x)\umbruch\\
    &=\frac1n\left(1+x^2(1+x)^2\right)(1-(n-1)x).
  \end{align*}
  This, however, is a direct consequence of
  $x\le\frac2n\frac n3=\frac23$ and
  \[x(1+x)^2\le\frac23\left(\frac53\right)^2=2\cdot\frac{25}{27}\le2.\]
  Observe, that the sign of $(1-(n-1)x)$ is irrelevant.  Our claim
  follows.
\end{proof}

We can now pass to the limit $\epsilon\downarrow0$ in Lemmata \ref{phi
  eps lower bd lem} and \ref{phi eps upper bd lem} to deduce
that \[\frac rn\le\phi(r)\le\frac rn +\left(\frac rn\right)^2\] for
$0\le r\le\frac n3$.  It follows that $\phi(r)$ can be extended
continuously to $[0,\infty)$. Moreover, this estimate implies
that \[\lim\limits_{r\downarrow0} \frac{\phi(r)}r =\frac1n.\]
Concerning the limit $\lim\limits_{r\downarrow0} \phi'(r)$, we proceed
exactly as in the end of Section \ref{approx pbls sec} and conclude
that the corresponding translator is of class $C^2$ at the origin.

\section{Existence via power series}
\label{power series sec}
 It is a pleasure to acknowledge that the idea to use this approach is
due to Heinrich Freist\"uhler. 
\begin{definition}
  Let 
  \[\Sigma^l_2 := \sum\limits_{\genfrac{}{}{0pt}{}{i,j\ge1}{i+j=l}}
    \frac1i \frac1j \qquad\text{and}\qquad\Sigma^l_3 :=
    \sum\limits_{\genfrac{}{}{0pt}{}{i,j,k\ge1}{i+j+k=l}} \frac1i
    \frac1j \frac1k.\] Define also corresponding suprema
  \[\Sigma_2 :=\sup\limits_{l\in\N} \Sigma^l_2\qquad\text{and}\qquad
    \Sigma_3 :=\sup\limits_{l\in\N} \Sigma^l_3.\]
\end{definition}

We start with an estimate for $\Sigma_2^{k-1}$. 
\begin{lemma}
\label{sum 2 lem}
For $k\in\N$, $k\ge3$, we have
$$\Sigma_2^{k-1}=\sum_{\substack{i,j\ge1 \\
    i+j=k-1}}\frac{1}{i}\frac{1}{j}\equiv
\sum_{i=1}^{k-2}\frac{1}{i}\frac{1}{k-1-i} \leq \frac{2\log(k-2)}{k-1}
+\frac{2}{k-2}.$$
\end{lemma}
\begin{proof}
  We use the monotonicity of $x\mapsto x\cdot(k-1-x)$ for
  $x<\frac{k-1}2$ and $x>\frac{k-1}2$ and obtain
\begin{align*}
  \sum_{i=1}^{k-2}\frac{1}{i}\frac{1}{k-1-i}
  &\leq \int\limits_1^{k-2}\frac{1}{x(k-1-x)}\,\, \text{d}x
    +\frac{2}{k-2}\\ 
  &= \frac{1}{k-1}\int\limits_1^{k-2}\frac{1}{x}+\frac{1}{k-1-x}\,\,
    \text{d}x +\frac{2}{k-2}\\ 
  &= \frac{1}{k-1}\left[\log(x)-\log(k-1-x) \right]\big|_1^{k-2}
    +\frac{2}{k-2}\\ 
  &= \frac{2\log(k-2)}{k-1} +\frac{2}{k-2}
\end{align*}
as claimed. 
\end{proof}

\begin{lemma}
  \label{Sigma 2 est one lem}
  We have $\Sigma_2=1$. 
\end{lemma}
\begin{proof}
  The first terms are easily computed as
  \[\Sigma^0_2=\Sigma^1_2=0,\quad
    \Sigma^2_2=\frac11\frac11=1\quad\text{and}\quad
    \Sigma^3_2=\frac11\frac12+\frac12\frac11 =1.\]
  We will now prove by induction on $l$ that $\Sigma^l_2\le1$ for
  every $l\in\N$. \par Let us first consider the ``odd'' case and show
  that for any $l\in\N$, there holds $\Sigma^{2l+2}_2\le1$ if
  $\Sigma^{2l+1}_2\le1$. We may assume that $l\ge1$ and obtain
  \begin{align*}
    \Sigma^{2l+1}_2
    &=2\frac11\frac1{2l} +2\frac12\frac1{2l-1} +\ldots
      +2\frac1{l-1}\frac1{l+2} +2\frac1l\frac1{l+1}
      \intertext{and}
      \Sigma^{2l+2}_2
    &=2\frac11\frac1{2l+1} +2\frac12\frac1{2l} +\ldots
      +2\frac1{l-1}\frac1{l+3} +2\frac1l\frac1{l+2}
      +\frac1{l+1}\frac1{l+1}. 
  \end{align*}
  The first $l$ terms differ by the factors
  \[\frac{2l}{2l+1} > \frac{2l-1}{2l} > \ldots > \frac{l+2}{l+3} >
    \frac{l+1}{l+2}.\] Observe, that each summand is
  positive. Therefore, we obtain the estimate
  \begin{align*}
    \Sigma^{2l+2}_2
    &\le\frac{2l}{2l+1} \Sigma^{2l+1}_2
      +\frac1{l+1}\frac1{l+1}\umbruch\\
    &\le\frac{2l}{2l+1} +\frac1{(l+1)^2}
      \intertext{by our induction hypothesis}
    &=1-\frac1{2l+1}+\frac1{(l+1)^2} \overset!\le1.
  \end{align*}
  This claim is equivalent to $2l+1\le(l+1)^2$ and hence obviously
  true. \par In the ``even'' case, we proceed similarly. We may once
  again assume that $l\ge1$ and get
  \begin{align*}
    \Sigma^{2l}_2
    &=2\frac11\frac1{2l-1} +2\frac12\frac1{2l-2} +\ldots
      +2\frac1{l-1}\frac1{l+1} +\frac1l\frac1l,\umbruch\\
    \Sigma^{2l+1}_2
    &=2\frac11\frac1{2l} +2\frac12\frac1{2l-1} +\ldots
      +2\frac1{l-1}\frac1{l+2} +2\frac1l\frac1{l+1},\umbruch\\
    \frac{2l-1}{2l}
    &>\frac{2l-2}{2l-1} >\ldots >\frac{l+1}{l+2} > \frac
      l{l+1},\umbruch\\
    \Sigma^{2l+1}_2
    &\le\frac{2l-1}{2l} \Sigma^{2l}_2 +\frac1l\frac1{l+1}\umbruch\\
    &\le1-\frac1{2l}+\frac1{l(l+1)} \overset!\le1.
  \end{align*}
  This is fulfilled for any $l\in\N$ such that $2l\le l(l+1)$. \par
  Hence, our claim follows. 
\end{proof}
 
For similar sums with three factors, we have the estimate
\begin{lemma}
  \label{sum 3 lem}
  For $l\in\N$, $l\ge4$, we have
  \begin{align*}
    \Sigma_3^{l-1}
    &=\sum_{\substack{i,j,k\ge1 \\
    i+j+k=l-1}}\frac{1}{i}\frac{1}{j}\frac{1}{k}\equiv
    \sum_{i=1}^{l-3}\sum_{j=1}^{l-2-i}
    \frac{1}{i}\frac{1}{j}\frac{1}{l-1-i-j}\umbruch\\  
    &\leq \frac{4\log(l-2)\log(l-3)}{l-1}+\frac{5\log(l-3)}{l-2}
    +\frac{4}{l-3}.
  \end{align*}
\end{lemma}
\begin{proof}
  We have for any $l\in\N$, $l\ge4$,
  \begin{align*}
    &\equad\sum_{i=1}^{l-3}\sum_{j=1}^{l-2-i}
      \frac{1}{i}\frac{1}{j}\frac{1}{l-1-i-j}\\  
    &=\sum_{i=1}^{l-3}\frac{1}{i}\sum_{j=1}^{l-2-i}
      \frac{1}{j}\frac{1}{l-1-i-j} \\ 
    &\leq \sum_{i=1}^{l-3}\frac{1}{i}\left( \frac{2\log(l-i-2)}{l-i-1}
      +\frac{2}{l-i-2} \right) 
      \intertext{in view of Lemma \ref{sum 2 lem}}
    &= 2\sum\limits_{i=1}^{l-3}\frac{1}{i} \frac{\log(l-i-2)}{l-i-1}
      +2\sum\limits_{i=1}^{l-3}\frac{1}{i}\frac{1}{l-i-2}\\ 
    &\overset{(*)}\le 2\log(l-3)\sum\limits_{i=1}^{l-4}\frac{1}{i}
      \frac{1}{l-i-1} 
      +2\sum\limits_{i=1}^{l-3}\frac{1}{i}\frac{1}{l-i-2}\\ 
    &= 2\log(l-3)\left[\sum\limits_{i=1}^{l-2} \frac1i
      \frac{1}{l-i-1}-\frac1{l-2}-\frac1{l-3}\frac12\right]
      +2\sum\limits_{i=1}^{l-3}\frac{1}{i}\frac{1}{l-i-2}\\ 
    &\le
      2\log(l-3)\left[\frac{2\log(l-2)}{l-1}
      +\frac{2}{l-2}-\frac1{l-2}-\frac1{l-3}\frac12\right]   
      +2\left[\frac{2\log(l-3)}{l-2}+\frac2{l-3}\right],
      \intertext{where we have used
      Lemma \ref{sum 2 lem} with $k=l$ and $k=l-1$,}  
    &= 2\log(l-3)\left[\frac{2\log(l-2)}{l-1}
      +\frac{l-4}{2(l-2)(l-3)}\right]  
      +2\left[\frac{2\log(l-3)}{l-2}+\frac2{l-3}\right].
  \end{align*}
  In view of $\frac{l-4}{l-3}\le1$, we deduce the claimed estimate.
  \par We wish to remark that the inequality $(*)$ is a natural candidate
  for further improvements.
\end{proof}

\begin{corollary}
  \label{drei sum 2 absch cor}
  For any $l\in\N$, $l\ge3$, we have
  \[\Sigma_3^l
    =\sum\limits_{\substack{i,j,k\ge1\\i+j+k=l}}
    \frac1i\frac1j\frac1k\le\Sigma_3\le 2.\]
\end{corollary}
\begin{proof}
  We have to show that $\Sigma_3^l\le 2$ for every $l\ge3$.  We divide
  the proof into two parts.\par
  \textbf{Large values of $l$:}
  First, we address large values of $l$,
  then we will give a computer-assisted argument for small values of
  $l$. Assume first that $l\ge e^6$. Then Lemma \ref{sum 3 lem}
  implies that
  \begin{align*}
    \Sigma_3^l
    &\le \frac{4\log(l-1)\log(l-2)}{l}
      +\frac{5\log(l-2)}{l-1} +\frac{4}{l-2}\umbruch\\
    &\le\frac{4\log^2l}l+\frac{l}{l-1}\frac{5\log l}{l}
      +\frac{l}{l-2}\frac4l\umbruch\\
    &\le\frac{4\log^2l}l +\frac{6\log l}{l} +\frac5l\umbruch\\
    &\le\frac{6\log^2l}l.
  \end{align*}
  We claim that for all $x\ge e^6$, we have $x^{1/3}\ge\log x$. At
  $x=e^6$, we have
  $x^{1/3}=\left(e^6\right)^{1/3} =e^2\ge (2.5)^2=6.25 \ge 6
  =\log\left(e^6\right)=\log x$. To extend this inequality to all
  $x\ge e^6$, we compute derivatives and investigate, whether
  $\frac13x^{-2/3}\ge\frac1x$. This is equivalent to
  $x^{1/3}\ge3$ and hence fulfilled for all $x\ge e^6$. We employ this
  estimate and deduce for all
  $l\ge 500\ge 22^2\ge 2.8^6\ge e^6$ that
  \[\Sigma^l_3\le 6\cdot l^{-1/3} \le 6\cdot e^{-2} \le \frac
    6{2.7^2}\le \frac 67\le 2.\] \par\textbf{Small values of $l$:} For
  $l\le500$, we rely on computer algebra calculations. The following
  \texttt{python}-program using \texttt{SymPy} \cite{sympy} verifies
  that $\Sigma_2^l\le1$ and $\Sigma_3^l\le2$ for $l\le500$. Note that
  we have already proven the first inequality for any $l\in\N$ in
  Lemma \ref{Sigma 2 est one lem}.
\begin{verbatim}
from sympy import Rational
maximal = 501
Sigma2 = []
for l in range(0,maximal):
  total = 0
  for i in range(1,l):
    j = l - i
    total = total + Rational(1,i*j)
  if total > Rational(4,5):
    print("{}: {}".format(l, 1-total))
  Sigma2.append(total)

print()
  
for l in range(3,maximal):
  total = 0
  for i in range(1,l-1):
    total = total + Rational(1,i) * Sigma2[l-i]
  if total > Rational(9,5):
    print("{}: {}".format(l, 2-total))
\end{verbatim}
  A few remarks concerning the program.
  \begin{itemize}
  \item We use fractions like \texttt{Rational(4,5)} for $\frac45$
    from \texttt{sympy} to avoid rounding errors.
  \item The values of $\Sigma_2^l$ are stored in \texttt{Sigma2} to
    speed up the computation of $\Sigma_3^l$ using
    \[\Sigma_3^l =\sum\limits_{i=1}^{l-2} \frac1i\cdot\Sigma_2^{l-i}.\]
    We have tested that the results coincide with those that we get when
    we compute $\Sigma_3^l$ with two \texttt{for}-loops similarly to
    $\Sigma_2^l$.
  \item A loop like \texttt{for i in range(1,s):} iterates over the
    following indented lines with $i=1,2,3,\ldots,s-1$. It is
    important to know that $i=s$ is not included.
  \item With \texttt{append}, we append an element to a given list.     
  \item In order not to produce too much output, we only generate
    output when $\Sigma_2^l\ge\frac45$ or $\Sigma_3^l\ge\frac95$. In
    those cases, we print $l$ and $1-\Sigma_2^l$ or $2-\Sigma_3^l$,
    respectively. This is always non-negative and confirms that
    $\Sigma_2^l\le1$ and $\Sigma_3^l\le2$. Technically, this is done
    using \texttt{print} and \texttt{format} for a formatted output.
  \item It takes about 5 seconds to do these computations on an older
    desktop computer. 
\end{itemize}
We would have preferred a purely theoretical proof, but could not find
anything that was not a lot longer and more involved than this
program.
\end{proof}

\begin{definition}
  \label{al def}
  Define for $n\in\N$, $n\ge2$, $a_0:=1$, $a_1:=\frac{1}{n+2}$ and
  inductively
  $$a_l:= -\frac{n-1}{2l+n}\sum_{\substack{i,j,k\ge1 \\ i+j+k=l-1}}
  a_ia_ja_k+\frac{-2n+3}{2l+n}\sum_{\substack{i,j\ge1 \\ i+j=l-1}}a_i
  a_j+\frac{-n+3}{2l+n}a_{l-1}$$ for $l\geq 2$. 
\end{definition}

\begin{remark}
  Let us mention for the readers convenience that the next term
  according to Definition \ref{al def} is given by
  $a_2:=\frac{3-n}{(n + 4)(n + 2)}$.
\end{remark}

Definition \ref{al def} is motivated by the following algebraic lemma.
\begin{lemma}
  \label{ode pol approx lem}
  Let $\psi$ be a polynomial and $\phi(r)=\psi\left(\frac
    rn\right)$. Let $M\in\N$. Then
  \[\phi'(r)
    -\left(1+\phi^2(r)\right)\left(1-\frac{n-1}r\phi(r)\right)\] is a
  polynomial that has a zero of order $2M+2$ at the origin, or,
  equivalently,
  \[\phi'(r)-\left(1+\phi^2(r)\right)
    \left(1-\frac{n-1}r\phi(r)\right)\in O\left(r^{2M+2}\right)\] near
  $r=0$, if and only if
  \[\psi(r)=\sum\limits_{i=0}^M a_ir^{2i+1} +R_M(r)\cdot r^{2M+3},\]
  where the coefficients $a_i\in\R$ are as in Definition \ref{al def} 
  and $R_M$ is an arbitrary polynomial.
\end{lemma}

\begin{proof}
  We rewrite the
  equation \[\phi'(r)=\left(1+\phi^2(r)\right)
    \left(1-\frac{n-1}r\phi(r)\right)\] 
  as an equation for $\psi$. We have $\phi'(r)=\frac1n\psi'\left(\frac
    rn\right)$.
  Renaming arguments, we obtain
  \begin{align*}
    \psi'(r)
    &=n\phi'(nr)\umbruch\\
    &=n\left(1+\phi^2(nr)\right)
      \left(1-\frac{n-1}{nr}\phi(nr)\right)\umbruch\\  
    &=\left(1+\psi^2(r)\right)\left(n-\frac{n-1}r\psi(r)\right)\umbruch\\
    &\equiv \text{RHS}(n,r,\psi(r)).
  \end{align*}
  Our claim is now equivalent to showing that
  $\psi'(r)-\text{RHS}(n,r,\psi(r))$ is a polynomial that has a zero
  of order $2M+2$ at the origin if and only if the coefficients of
  $\psi$ in front of $1,r,\ldots,r^{2M+2}$ are as in the claim. \par
  Implicitly, our claim states that some even coefficients of $\psi$
  vanish. This has to be justified in our proof.  Therefore, we start
  with the more general ansatz
  \[\psi(r) =\sum\limits_{i\ge0} b_ir^i,\] where at most finitely many
  $b_i\in\R$ are different from zero. \par
  We wish to rewrite the condition on the order of the zero as a
  condition on the coefficients $b_i$, $1\le i\le M$. We have
  \begin{align*}
    &\equad \psi'(r) -\left(1+\psi^2(r)\right)
      \left(n-\frac{n-1}r\psi(r)\right)\umbruch\\
    &=\sum\limits_{i\ge1} ib_ir^{i-1} -\left(1+\sum\limits_{i,j\ge0}
      b_ib_jr^{i+j}\right) \left(n-(n-1)\sum\limits_{i\ge0}
      b_ir^{i-1}\right)\umbruch\\
    &=\sum\limits_{i\ge0} (i+1)b_{i+1}r^i -n\sum\limits_{i\ge0}
      \delta_{i0}r^i +(n-1)\sum\limits_{i\ge-1} b_{i+1}r^i\\
    &\equad
      -n\sum\limits_{d\ge0} \sum\limits_{\substack{i,j\ge0\\ i+j=d}}
    b_ib_jr^d +(n-1)\sum\limits_{d\ge-1}
    \sum\limits_{\substack{i,j,k\ge0 \\ i+j+k =d+1}} b_ib_jb_k r^d. 
  \end{align*}
  From here, we can read off conditions for a zero at $r=0$ of a
  certain order. First of all, this expression is a polynomial if and
  only if the coefficient in front of $r^{-1}$ vanishes,
  \[0=(n-1)b_0 +(n-1)b_0^3 =(n-1)b_0\left(1+b_0^2\right).\] Therefore,
  we will henceforth assume that $b_0=0$. \par The expression
  $\psi'-RHS$ has a zero of order $2M+2$ at the origin, if for all
  $d=0,1,\ldots,2M+1$, we have
  \begin{align*}
    0
    &=(d+1)b_{d+1} -n\delta_{d0} +(n-1)b_{d+1}
      -n\sum\limits_{\substack{i,j\ge0 \\ i+j=d}} b_ib_j
    +(n-1)\sum\limits_{\substack{i,j,k\ge0 \\ i+j+k=d+1}} b_ib_jb_k. 
  \end{align*}
  We will consider these conditions for $d=0$ and for $d\ge1$
  separately. We use $b_0=0$ and get for $d=0$
  \[0=1\cdot b_1-n+(n-1)b_1 =n\cdot b_1-n\] and thus $b_1=1$. For
  $1\le d\le 2M+1$, we obtain
  \begin{equation}
    \label{bd def eq}
    0=(d+n)\cdot b_{d+1} -n\sum\limits_{\substack{i,j\ge0 \\ i+j=d}}
    b_ib_j +(n-1)\sum\limits_{\substack{i,j,k\ge0\\ i+j+k=d+1}}
    b_ib_jb_k. 
  \end{equation}
  Computing a few $b_i$'s according to that formula by hand suggests
  that these conditions are only fulfilled for all $d$ if $b_i=0$ for
  even values of $i$. This follows also from geometric
  considerations. We are trying to find $U\colon[0,\infty)\to\R$ such
  that $x\mapsto U(|x|)$ is a rotationally symmetric function which is
  a translating solution to mean curvature flow at some fixed
  time. The symmetry implies that $U$ has to be an even
  function. Therefore, $\phi=U'$ and $\psi$ have to be odd. \par We
  now prove this by induction based on \eqref{bd def eq}. We have
  already seen that $b_0=0$. Consider some odd value of
  $1\le d\le2M+1$ and assume that $b_i=0$ for all even $i<d$. Now, we
  have a closer look at \eqref{bd def eq} and use it to compute
  $b_{d+1}$ for even $d+1$. As $d$ is odd, at least one index, $i$ or
  $j$, in the first sum has to be even. Hence, this sum
  vanishes. Similarly for the second sum. As $i+j+k=d+1$ is even, at
  least one index, $i$, $j$ or $k$, has to be even.  Moreover,
  summands can at most be nonzero if $i,j,k\ge1$ and hence
  $i,j,k\le d-1$.  Therefore also this second sum vanishes. We
  conclude that $b_i=0$ for all even $i$ with $i\le 2M+2$ if $b_0=0$
  and \eqref{bd def eq} is fulfilled for all odd $d\le 2M+1$. \par
  Now, we will consider \eqref{bd def eq} for $d=0,2,4,\ldots,2M$. We
  use that $b_i=0$ for even $i$ with $0\le i\le2M+2$ and rewrite
  \eqref{bd def eq} as
  \begin{align*}
    (2l+n)b_{2l+1}
    &=n\cdot\sum\limits_{\substack{i,j\ge0\\ 2i+1+2j+1=2l}}
    b_{2i+1}b_{2j+1}\\
    &\equad-(n-1)\cdot\sum\limits_{\substack{i,j,k\ge0\\
    2i+1+2j+1+2k+1=2l+1}} b_{2i+1}b_{2j+1}b_{2k+1}
  \end{align*}
  for $l=0,1,2,\ldots,M$. We set $a_i:=b_{2i+1}$ and obtain for
  $l=0,1,2,\ldots,M$
  \begin{align*}
    (2l+n)a_l
    &=n\sum\limits_{\substack{i,j\ge0\\ i+j=l-1}}a_ia_j
    -(n-1)\sum\limits_{\substack{i,j,k\ge0\\ i+j+k=l-1}} a_ia_ja_k. 
  \end{align*}
  Although this is not identical to the formula in Definition \ref{al
    def}, it is also an appropriate formula for computing the
  coefficients $a_i$ recursively. The formula in Definition \ref{al
    def}, however, incorporates a few cancellations that we will see
  in the next lines. They will later facilitate our estimates. \par
  Recall that $a_0=b_1=1$. For $l=1$, we obtain
  \begin{align*}
    (n+2)a_1
    &=n\cdot1-(n-1)\cdot1,\umbruch\\
    a_1
    &=\frac1{n+2}. 
  \end{align*}
  Further rewriting this identity yields for $l\ge2$
  \begin{align*}
    (2l+n)a_l
    &=n\sum\limits_{\substack{i,j\ge1\\ i+j=l-1}} a_ia_j +n\cdot2\cdot
    a_{l-1} -(n-1)\sum\limits_{\substack{i,j,k\ge1\\ i+j+k=l-1}}
    a_ia_ja_k\\
    &\equad -3(n-1)\sum\limits_{\substack{i,j\ge1\\ i+j=l-1}} a_ia_j
    -(n-1)\cdot3\cdot a_{l-1}\umbruch\\
    &=-(n-1)\sum\limits_{\substack{i,j,k\ge1\\ i+j+k=l-1}} a_ia_ja_k
    +(-2n+3)\sum\limits_{\substack{i,j\ge1\\ i+j=l-1}} a_ia_j
    +(-n+3)a_{l-1}.
  \end{align*}
  This yields precisely the condition in Definition \ref{al def}. 
  \par For the reverse direction, we essentially use the same
  calculations.  
\end{proof}

\begin{theorem}
  \label{14l thm}
  For $n\in\{2,3,4\}\subset\N$ we have the estimate
  \[|a_l|\leq \frac{1}{4l}\] for all $l\in \N$.
\end{theorem}

\begin{proof}
  We use induction on $l$ to prove our claim. We will check it
  individually for $l=0,1,2$. For $l=0$, it is obvious. For $l=1$, we
  have
  \begin{align*}
    a_1
    &=\frac1{n+2}\le\frac1{2+2}=\frac1{4\cdot1}
      \intertext{as claimed and for $l=2$, we obtain}
      |a_2|
    &=\frac{|n-3|}{(n+4)(n+2)}\le\frac1{6\cdot4}\le\frac1{4\cdot2}.
  \end{align*}
  Now, let $l\ge3$ and assume that our claim holds up to $l-1$. We use
  $\Sigma_2\le1$ from Lemma \ref{Sigma 2 est one lem} and
  $\Sigma_3\le2$ from Lemma \ref{drei sum 2 absch cor} to obtain
  \begin{align*}
    |a_l|
    &\leq \frac{n-1}{2l+n}\sum_{\substack{i,j,k\ge1 \\ i+j+k=l-1}}
    |a_i||a_j||a_k|+\frac{2n-3}{2l+n}\sum_{\substack{i,j\ge1 \\
    i+j=l-1}}|a_i| |a_j|+\frac{|n-3|}{2l+n}|a_{l-1}|\umbruch\\ 
    &\leq \frac{n-1}{2l+n}\frac{1}{64}\sum_{\substack{i,j,k\ge1 \\
    i+j+k=l-1}}
    \frac{1}{ijk}+\frac{2n-3}{2l+n}\frac{1}{16}\sum_{\substack{i,j\ge1
    \\ i+j=l-1}}\frac{1}{ij}+\frac{|n-3|}{2l+n}\frac{1}{4(l-1)}\umbruch\\ 
    &\leq
      \frac{1}{2l+n}\cdot\frac12\left(\frac{n-1}{32}\cdot\Sigma_3
      +\frac{2n-3}{8}\cdot\Sigma_2+\frac{|n-3|}{2(l-1)}\right)\umbruch\\ 
    &\leq
      \frac{1}{2l+n}\cdot\frac12\left(\frac{n-1}{16}
      +\frac{2n-3}{8}+\frac{|n-3|}{2(l-1)}\right)\umbruch\\
    &=
      \frac{1}{2l+n}\cdot\frac12\left(\frac{5n-7}{16}
      +\frac{|n-3|}{2(l-1)}\right).
  \end{align*}
  If $n\in\{2,3\}$, we use $l\ge3$ and deduce
  \[|a_l|
    \le\frac1{4l}\left(\frac{8}{16}+\frac1{2\cdot2}\right)
    \le\frac1{4l}\cdot\frac34 \le\frac1{4l}.\]
  Similarly, we have for $n=4$ and $l\ge4$
  \[|a_l|\le\frac1{4l}\left(\frac{13}{16}+\frac1{2\cdot3}\right)
    =\frac1{4l}\cdot\frac{39+8}{48}
    =\frac1{4l}\cdot\frac{47}{48}\le\frac1{4l}.\]
  In the remaining case $n=4$ and $l=3$, we calculate
  \[|a_3|\le\frac1{6+4}\cdot\frac12
    \left(\frac{13}{16}+\frac1{2\cdot2}\right) 
    =\frac1{20}\cdot\frac{17}{16}\le\frac1{16}\le\frac1{12}
    =\frac1{4l}.\] 
  As we have established the claimed inequality in all cases, our
  Theorem follows.   
\end{proof}

\begin{remark}
  \neueZeile
  \begin{enumerate}[label=(\roman*)]
  \item For $n\ge5$, the inequality
    \[\frac1{2l+n}\cdot\frac12
      \left(\frac{5n-7}{16}+\frac{n-3}{2(l-1)}\right)\le\frac1{4l}\]
    becomes false for $l\in\N$ sufficiently large as the term in
    parantheses is strictly larger than one.
  \item Naive computer computations suggest that the decay rate
    obtained in Theorem \ref{14l thm} is not at all optimal. Our
    \texttt{python}-calculations yield for $l<500$
    \[|a_l|\le e^{-\lambda(n)\cdot l}\]
    with $\lambda(2)>1.09$, $\lambda(n)>0.50$ for $n\le20$ and
    $\lambda(n)>0.34$ for $n\le50$.
  \item Such exponential estimates, however, would also not guarantee
    in infinite radius of convergence for our power series.
  \item Moreover, it is not clear to us how to inductively show such
    an estimate based on Definition \ref{al def}.     
  \end{enumerate}
\end{remark}

\begin{theorem}
  Let $n\in\{2,3,4\}$. Then the power series
  \[\psi(r):=\sum\limits_{i=0}^\infty a_i r^{2i+1}\]
  with coefficients defined in Definition \ref{al def} converges on
  $(-1,1)$ and $\phi(r):=\psi(r/n)$ converges on $(-n,n)$ and solves
  \[\phi'(r)=\left(1+\phi^2(r)\right)\left(1-\frac{n-1}r\phi(r)\right).\]
\end{theorem}
\begin{proof}
  As $|a_i|\le1$ for all $i\in\N$, the convergence radius of $\psi$ is
  at least $1$ according to the root test for power series. This means
  that \[\psi_k(r):=\sum\limits_{i=0}^k a_i r^{2i+1}\] converges
  locally uniformly on $(-1,1)$ to a function $\psi$. We also obtain
  locally uniform convergence of the other terms in the differential
  equation, in particular $\psi_k'\to\psi'$, $\psi_k^2\to\psi^2$ and
  $\frac{\psi_k(r)}r\to\frac{\psi(r)}r$.
  Consider
  \[\Psi(\psi,r):=\psi'(r)-\left(1+\psi^2(r)\right)
    \left(n-\frac{n-1}r\psi(r)\right).\] We see that
  $r\mapsto\Psi(\psi,r)$ is analytic and also converges on
  $(-1,1)$. Our claim is that $\Psi(\psi,r)=0$. As $\Psi(\psi,r)$ can
  be written as a power series $\sum\limits_{i=0}^\infty b_ir^{2i}$,
  it suffices to show that $b_i=0$ for any $i\in\N$.  Similarly, we
  can write $\Psi(\psi_k,r) =\sum\limits_{i=0}^{N(k)}
  b_{k,i}r^{2i}$. By construction, see Lemma \ref{ode pol approx lem},
  we have $b_{k,i}=0$ for $k>i$. Since $b_i$ depends only on $a_l$,
  $l\le i$, we infer $b_i=b_{k,i}=0$ for $k>i$.\qedhere
\end{proof}

\begin{remark}
  Numerical calculations suggest that the radius $r$ of convergence
  for $\phi$ is indeed finite and approximately given by the values in
  the following tabular.
  
  \begin{center}
    \begin{tabular}{|c||c|c|c|c|c|c|c|c|c|}
      \hline
      $n$ & 2 & 3 & 4 & 5 & 6 &  7 &  8 &  9 & 10 \\\hline\hline
      $r$ &3.4&4.9&6.3&7.6&8.9&10.2&11.4&12.7&13.9\\\hline
    \end{tabular}
  \end{center}
\end{remark}

\def\emph#1{\textit{#1}}
\bibliographystyle{amsplain}
\def\weg#1{} \def\unterstrich{\underline{\rule{1ex}{0ex}}} \def\cprime{$'$}
  \def\cprime{$'$} \def\cprime{$'$} \def\cprime{$'$}
\providecommand{\bysame}{\leavevmode\hbox to3em{\hrulefill}\thinspace}
\providecommand{\MR}{\relax\ifhmode\unskip\space\fi MR }
\providecommand{\MRhref}[2]{%
  \href{http://www.ams.org/mathscinet-getitem?mr=#1}{#2}
}
\providecommand{\href}[2]{#2}

\end{document}